\documentclass[11pt]{article}
\usepackage{amsfonts,mathrsfs,amssymb,amsthm,mathptm}
\usepackage{amsmath,amscd}
\usepackage{mathptm,pslatex}
\usepackage{color}
\usepackage[dvips]{graphicx}
\usepackage[all]{xy}
\usepackage{graphicx}
\usepackage{fancyhdr}
\usepackage{url}

\begin{document}
\newtheorem{defi}{Definition}[section]
\newtheorem{exam}[defi]{Example}
\newtheorem{Prop}[defi]{Proposition}
\newtheorem{Theo}[defi]{Theorem}
\newtheorem{Lem}[defi]{Lemma}
\newtheorem{coro}[defi]{Corollary}
\theoremstyle{definition}
\newtheorem{rem}[defi]{Remark}
\newtheorem{ques}[defi]{Question}

\newcommand{\add}{{\rm add}}
\newcommand{\con}{{\rm con}}
\newcommand{\gd}{{\rm gl.dim}}
\newcommand{\sd}{{\rm st.dim}}
\newcommand{\sr}{{\rm sr}}
\newcommand{\dm}{{\rm dom.dim}}
\newcommand{\cdm}{{\rm codomdim}}
\newcommand{\tdim}{{\rm dim}}
\newcommand{\E}{{\rm E}}
\newcommand{\Mor}{{\rm Morph}}
\newcommand{\End}{{\rm End}}
\newcommand{\rank}{{\rm rank}}
\newcommand{\PSL}{{\rm PSL}}
\newcommand{\GL}{{\rm GL}}
\newcommand{\ind}{{\rm ind}}
\newcommand{\rsd}{{\rm res.dim}}
\newcommand{\rd} {{\rm rd}}
\newcommand{\ol}{\overline}
\newcommand{\overpr}{$\hfill\square$}
\newcommand{\rad}{{\rm rad}}
\newcommand{\soc}{{\rm soc}}
\renewcommand{\top}{{\rm top}}
\newcommand{\pd}{{\rm pdim}}
\newcommand{\id}{{\rm idim}}
\newcommand{\fld}{{\rm fdim}}
\newcommand{\Fac}{{\rm Fac}}
\newcommand{\Gen}{{\rm Gen}}
\newcommand{\fd} {{\rm fin.dim}}
\newcommand{\Fd} {{\rm Fin.dim}}
\newcommand{\Pf}[1]{{\mathscr P}^{<\infty}(#1)}
\newcommand{\DTr}{{\rm DTr}}
\newcommand{\cpx}[1]{#1^{\bullet}}
\newcommand{\D}[1]{{\mathscr D}(#1)}
\newcommand{\Dz}[1]{{\mathscr D}^+(#1)}
\newcommand{\Df}[1]{{\mathscr D}^-(#1)}
\newcommand{\Db}[1]{{\mathscr D}^b(#1)}
\newcommand{\C}[1]{{\mathscr C}(#1)}
\newcommand{\Cz}[1]{{\mathscr C}^+(#1)}
\newcommand{\Cf}[1]{{\mathscr C}^-(#1)}
\newcommand{\Cb}[1]{{\mathscr C}^b(#1)}
\newcommand{\Dc}[1]{{\mathscr D}^c(#1)}
\newcommand{\K}[1]{{\mathscr K}(#1)}
\newcommand{\Kz}[1]{{\mathscr K}^+(#1)}
\newcommand{\Kf}[1]{{\mathscr  K}^-(#1)}
\newcommand{\Kb}[1]{{\mathscr K}^b(#1)}
\newcommand{\DF}[1]{{\mathscr D}_F(#1)}

\newcommand{\Kac}[1]{{\mathscr K}_{\rm ac}(#1)}
\newcommand{\Keac}[1]{{\mathscr K}_{\mbox{\rm e-ac}}(#1)}

\newcommand{\modcat}{\ensuremath{\mbox{{\rm -mod}}}}
\newcommand{\cmodcat}{\ensuremath{\mbox{{\rm -comod}}}}
\newcommand{\Modcat}{\ensuremath{\mbox{{\rm -Mod}}}}
\newcommand{\ires}{\ensuremath{\mbox{{\rm ires}}}}
\newcommand{\Stires}{\ensuremath{\mbox{{\rm Stires}}}}
\newcommand{\Stpres}{\ensuremath{\mbox{{\rm Stpres}}}}
\newcommand{\Spec}{{\rm Spec}}

\newcommand{\stmc}[1]{#1\mbox{{\rm -{\underline{mod}}}}}
\newcommand{\Stmc}[1]{#1\mbox{{\rm -{\underline{Mod}}}}}
\newcommand{\prj}[1]{#1\mbox{{\rm -proj}}}
\newcommand{\inj}[1]{#1\mbox{{\rm -inj}}}
\newcommand{\Prj}[1]{#1\mbox{{\rm -Proj}}}
\newcommand{\Inj}[1]{#1\mbox{{\rm -Inj}}}
\newcommand{\PI}[1]{#1\mbox{{\rm -Prinj}}}
\newcommand{\GP}[1]{#1\mbox{{\rm -GProj}}}
\newcommand{\GI}[1]{#1\mbox{{\rm -GInj}}}
\newcommand{\gp}[1]{#1\mbox{{\rm -Gproj}}}
\newcommand{\gi}[1]{#1\mbox{{\rm -Ginj}}}

\newcommand{\opp}{^{\rm op}}
\newcommand{\otimesL}{\otimes^{\rm\mathbb L}}
\newcommand{\rHom}{{\rm\mathbb R}{\rm Hom}\,}
\newcommand{\pdim}{\pd}
\newcommand{\Hom}{{\rm Hom}}
\newcommand{\Coker}{{\rm Coker}}
\newcommand{ \Ker  }{{\rm Ker}}
\newcommand{ \Cone }{{\rm Con}}
\newcommand{ \Img  }{{\rm Im}}
\newcommand{\Ext}{{\rm Ext}}
\newcommand{\StHom}{{\rm \underline{Hom}}}
\newcommand{\StEnd}{{\rm \underline{End}}}
\newcommand{\KK}{I\!\!K}
\newcommand{\gm}{{\rm _{\Gamma_M}}}
\newcommand{\gmr}{{\rm _{\Gamma_M^R}}}

\def\demo{{\bf Proof}\hskip10pt}

\def\g{\gamma} \def\d{\delta} \def\a{\alpha}
\def\s{\sigma}  \def\om{\omega}  \def\ld{\lambda}
\def\D{\Delta}
\def\si{\Sigma} \def\O{\Omega}
\def\G{\Gamma} \def\GG{{\cal G}} \def\XX{{\cal X}} \def\MM{{\cal M}} \def\lrr{\lg r\rg } \def\ogg{\overline {\GG}}
\def\og{\overline G} \def\oh{\overline H} \def\oc{\overline C}
 \def\oq{\overline Q}

 \def\oa{\overline A}
  \def\ob{\overline B}  \def\ol{\overline L} \def\om{\overline M}
\def\on{\overline N} \def\op{\overline P} \def\os{\overline S}
\def\ot{\overline T} \def\ok{\overline K} \def\ov{\overline V}
\def\od{\overline D} \def\oi{\overline I} \def\oj{\overline J}
\def\o1{\overline 1} \def\olh{\overline h}

\def\o{\overline}   \def\olr{\overline r} \def\oll{\overline \ell} \def\olt{\overline t }

\def\di{\bigm|} \def\lg{\langle} \def\rg{\rangle}

\def\vez{\varepsilon}\def\bz{\bigoplus}  \def\sz {\oplus}
\def\epa{\xrightarrow} \def\inja{\hookrightarrow}

\newcommand{\lra}{\longrightarrow}
\newcommand{\llra}{\longleftarrow}
\newcommand{\lraf}[1]{\stackrel{#1}{\lra}}
\newcommand{\llaf}[1]{\stackrel{#1}{\llra}}
\newcommand{\ra}{\rightarrow}
\newcommand{\dk}{{\rm dim_{_{k}}}}

\newcommand{\holim}{{\rm Holim}}
\newcommand{\hocolim}{{\rm Hocolim}}
\newcommand{\colim}{{\rm colim\, }}
\newcommand{\limt}{{\rm lim\, }}
\newcommand{\Add}{{\rm Add }}
\newcommand{\Prod}{{\rm Prod }}
\newcommand{\pres}{\ensuremath{\mbox{{\rm pres}}}}
\newcommand{\app}{{\rm app }}
\newcommand{\Tor}{{\rm Tor}}
\newcommand{\Cogen}{{\rm Cogen}}
\newcommand{\Tria}{{\rm Tria}}
\newcommand{\Loc}{{\rm Loc}}
\newcommand{\Coloc}{{\rm Coloc}}
\newcommand{\tria}{{\rm tria}}
\newcommand{\Con}{{\rm Con}}
\newcommand{\Thick}{{\rm Thick}}
\newcommand{\thick}{{\rm thick}}
\newcommand{\Sum}{{\rm Sum}}
\def\Mon{\hbox{\rm Mon}}
\def\Aut{\hbox{\rm Aut}}
\def\SL{\hbox{\rm SL}}
\newcommand{\PGL}{{\rm PGL}}
\def\Syl{\hbox{\rm Syl}}
\def\char{\hbox{\rm \,char\,}}

{\Large \bf
\begin{center}
Permutation groups and symmetric Hecke algebras
\end{center}}

\medskip
\centerline{ \textbf{Jiawei He} and  \textbf{Xiaogang Li}$^*$}

\medskip

\renewcommand{\thefootnote}{\alph{footnote}}
\setcounter{footnote}{-1} \footnote{ $^*$ Corresponding author.
Email: 2200501002@cnu.edu.cn.}
\renewcommand{\thefootnote}{\alph{footnote}}
\setcounter{footnote}{-1}
\renewcommand{\thefootnote}{\alph{footnote}}
\setcounter{footnote}{-1}

\begin{abstract}
The endomorphism algebras of the permutation modules for  transitive permutation groups, known as Hecke algebras, are fundamental objects in representation theory. While group algebras are known to be symmetric over any field, it is natural to ask whether this property extends to Hecke algebras. To study this, we introduce the new concepts of $p$-$S$-permutation groups (for a prime $p$) and $S$-permutation groups. A \emph{ $p$-$S$-permutation group} is a transitive permutation group whose associated Hecke algebra is symmetric over every field of characteristic $p$. An \emph{ $S$-permutation group} is a transitive permutation group that is a $p$-$S$-permutation group for all primes $p$. In this paper, we study Hecke algebras from a group-theoretical perspective and we show that several classes of permutation groups are $p$-$S$-permutation groups and $S$-permutation groups in our sense. This result represents a substantial extension of earlier work by Li and He. (Transform Groups, 30(4), 2025), and reframes the question of determining when the algebra \(\End_{KG}(K\Omega)\) is symmetric within a more general theoretical framework.
\end{abstract}

\section{Introduction}

The theory of permutation groups has a long and distinguished history, which dates back at least to the age of Galois, when he employed permutations of roots to solve the problem of solvability by radicals for polynomial equations. The well-known Cayley's theorem shows that every abstract group can be realized as a permutation group, so that there is no difference between abstract groups and permutation groups, from an algebraic structural point of view. But permutation group theory is of interest in its own right because of the notions of  transitivity, primitivity, and point stabilizers which have no counterparts in the abstract theory. Furthermore, permutation groups play an important role in other branches of mathematics, including graph theory, Galois theory, and function theory. Closely connected to permutation group theory is the notion of a $G$-set, a set $\Omega$ equipped with an action of a group $G$. The induced permutation group on $\Omega$ is denoted $G^\Omega$. For the permutation module $K\Omega$ over the group algebras $KG$ and $KG^\Omega$, one has $\End_{KG}(K\Omega)\cong \End_{KG^\Omega}(K\Omega)$. If the action of $G$ on $\Omega$ is transitive with  point stabilizer $H$, then $\End_{KG}(K\Omega)$ can be identified naturally with the  Hecke algebra $\mathcal{H}_K(G,H)$ (see Section \ref{sect2} for definition of this notation). This establishes a relation between Hecke algebras and endomorphism algebras of  permutation modules for transitive permutation groups.

\smallskip
The algebra $\End_{KG}(K\Omega)$ has its applications in many branches of mathematics. Using this family of algebras, Schur \cite{SC} successfully established that any primitive permutation group containing a regular subgroup of composite order must be $2$-transitive. Within combinatorics, $\End_{KG}(K\Omega)$ is studied as a centralizer ring in the settings of both associative schemes \cite{W64} and coherent configurations \cite{SC}; such investigations form an indispensable part of the analysis of graphs and their associated automorphism groups.
In the representation theory of finite groups, Alperin \cite{Alp} proposed studying the endomorphism algebra $\End_{KG}(K(G/P))$ as an approach to attack Alperin's Weight Conjecture, where $P$ denotes a Sylow $p$-subgroup of $G$. Building on this framework, Green \cite{Gr} constructed a bijection between the simple modules lying in the socles of $K(G/P)$ and simple modules of $\End_{KG}(K(G/P))$, under the hypothesis that $\End_{KG}(K(G/P))$ is self-inejctive. Based on Green's results, Cabanes \cite{Ca} verified the Alperin Weight Conjecture for all finite groups $G$ with a split $BN$-pair. In the classical theory of special functions, and to the theory of Riemannian symmetric spaces in differential geometry, finding a Gelfand pair $(G,H)$ (consisting of a group $G$ and a subgroup $H$) is an on-going research program, which is equivalent to finding a subgroup $H$ in a given group $G$ such that the Hecke algebras $\End_{KG}(K(G/H))$ is commutative for any field $K$. Despite the effort to develop an analogue of Brauer's Main theorems for modular group algebras by L. Scott in \cite{Sco} and the large literatures on the study of classical Hecke algebras and group algebras, the representation theory of abstract Hecke algebra $\End_{KG}(K\Omega)$ is far from known. It is worthy to note that group algebras over an arbitrary field are always symmetric algebras, a class of algebras of fundamental importance in representation theory. With the aim of determining the extent to which this nice property of group algebras can be generalized to the Hecke algebra $\End_{KG}(K\Omega)$, an important goal of the present work is to characterize those transitive permutation groups for which their corresponding Hecke algebras $\End_{KG}(K\Omega)$ are symmetric.
\smallskip

In a recent work \cite{AKOM}, A. Kleshchev explored the maximally symmetric subalgebras of symmetric algebras, providing a novel perspective from which to analyze the substructure of $\End_{KG}(K\Omega)$ when viewed as a subobject within the category of symmetric algebras. In \cite{KCEP}, Kevin Coulembier and hia coauthors investigated finite symmetric algebras in tensor categories and Verlinde categories associated with algebraic groups, thereby developing a systematic theoretical framework for studying symmetry properties of $\End_{KG}(K\Omega)$ in considerably more general categorical settings.
Whether \(\End_{KG}(K\Omega)\) itself carries the structure of a symmetric algebra depends heavily on the characteristic of the field \(K\). Initial results in this direction were obtained in \cite{LH}, where it was proved that \(\End_{KG}(K\Omega)\) is a symmetric algebra over a field \(K\) of characteristic \(p\) provided one of the following conditions holds:

\medskip
$(1)$ $G$ is a primitive permutation group, $|\Omega|<6p$ and $( p,|\Omega|)\neq (5,25)$;
\medskip

$(2)$ $G$ is a quasi-primitive permutation group and $|\Omega|<4p$.

\medskip
In order to bring this research further and study their relation from a more general way, we introduce the following definition.

\begin{defi}
{\rm A transitive permutation group $G$ on a set $ \Omega $ is called a \emph{$ p $-$ S $-permutation group} (for a prime $ p $) if $ \operatorname{End}_{KG}(K\Omega) $ is a symmetric algebra for every field $ K $ of characteristic $ p $. If $ G $ is a $ p $-$ S $-permutation group for all primes $ p $, we simply call $G$ an \emph{$S$-permutation group}.}
\end{defi}

The main result of \cite{LH} and its corollary show that a primitive permutation group of degree $n$ is a $p$-$S$-permutation group whenever $p>\frac{n}{6}$ and $(p,n)\neq (5,25)$, and that a quasi-primitive permutation group of degree $n$  is a $p$-$S$-permutation group whenever $p>\frac{n}{4}$. Typical examples of $S$-permutation groups include $2$-transitive permutation groups and regular permutation groups. Recall that a regular permutation group is a transitive permutation group with trivial point stabilizer.

\smallskip

On the other hand, permutation group theory has yielded a wealth of remarkable results. A well-known example is the O'Nan-Scott theorem, which divides all primitive groups into several families. There are fruitful literatures on characterizing permutation groups with limitations on their ranks, degrees or point stabilizers. These results hold great potential for application in our research. Therefore, we pay special attention to those permutation groups that have been extensively studied. In this paper, we study Hecke algebras from a group theoretical point of view and give several sufficient criteria for a permutation group to be a $p$-$S$-permutation group or an $S$-permutation groups.

\smallskip
Before stating our main results, we recall some basic notation on permutation groups.
Let $G \leq \mathrm{Sym}(\Omega)$. We say that the group $G$ is:

\emph{transitive:}  if $G$ has exactly one orbit on $\Omega$;

\emph{$\frac{1}{2}$-transitive }(or \emph{half-transitive}): if all $G$-orbits on $\Omega$ have the same cardinality;

\emph{$\frac{3}{2}$-transitive:} if $G$ is transitive and the point stabilizer $G_\alpha$ is half-transitive on $\Omega \setminus \{\alpha\}$;

\emph{2-transitive:} if $G_\alpha$ is transitive on $\Omega\setminus \{\alpha\}$, or equivalently, $G$ acts transitively on the set of ordered pairs of distinct elements of $\Omega$.

\medskip
In the context of classifying $\frac{3}{2}$-transitive permutation groups, recent advances have established that such groups decompose into six families: 2-transitive groups, Frobenius groups, 1-dimensional affine groups, affine solvable subgroups of ${\rm AGL}(2, q)$, special projective linear groups ${\rm PSL}(2, q)$, and projective semilinear groups ${\rm P\Gamma L}(2, q)$, with $q = 2^p$ for a prime $p$. We will show that any $\frac{3}{2}$-transitive permutation group is necessarily an \( S \)-permutation group.
\medskip

The following is our main result (for unexplained notation, we refer the reader to Section 2).
\begin{Theo}\label{main}
Let $G$ be a transitive permutation group on a set $\Omega$ of cardinality $n$. Then the following statements hold.
\begin{enumerate}
\item[\rm (1)] Let $p$ be a prime. Then $G$ is a $p$-$S$-permutation group if at least one of the following conditions is satisfied:

$(i)$ all subdegrees of $G$ are coprime to $p$;

$(ii)$ $p\nmid \frac{|G|}{nm}$ for any non-trivial subdegree $m$ of $G$;

$(iii)$ $G$ has an abelian regular $p'$-subgroup;

$(iv)$ $n<2p$;

$(v)$ $p\nmid n$ and $G$ is of rank $3$;

$(vi)$ $\Omega=G(q)/\mathcal{P}$ and $G$ is the permutation group induced by the right multiplication of $G(q)$ on $\Omega$, where $G(q)$  is finite group with a split $BN$-pair at characteristic $q\neq p$ and $\mathcal{P}$ is a parabolic subgroup of $G(q)$ such that $p\nmid \frac{|\mathcal{P}|}{|B|}$.

\item[\rm (2)] The group  $G$ is an $S$-permutation group if it satisfies one of the following:

$(1)$ $G$ is a $\frac{3}{2}$-transitive permutation group;

$(2)$ $G$ has a cyclic regular subgroup whose order is a product of two distinct primes;

$(3)$ $G$ has an abelian regular subgroup whose order is coprime to every subdegree of $G$;

$(4)$ $G$ is a rank $3$ permutation group of subdegrees $1\leq a<b$ and $\gcd(a+b+1,ab,\lambda(G), b-\frac{\lambda(G)b}{a})=1$ (see Remark \ref{rank-3} for the definition of $\lambda(G)$);

$(5)$ $G$ is a dihedral group of order not divisible by $8$;

$(6)$ $\Omega=G(q)/B$ and $G$ is the permutation group induced by the right multiplication of $G(q)$ on $\Omega$, where $G(q)$ is finite group with a split $BN$-pair at characteristic $q$ satisfying $\mathcal{W}_G\ncong S_3$ is of rank 2.
\end{enumerate}
\end{Theo}

This paper is organized as follows. In Section \ref{sect2}, we will give basic notation, review basic definitions and terminology, and collect several preliminary lemmas required for the subsequent proofs. In Section \ref{pf}, we will give a proof of Theorem \ref{main}. In Section \ref{exam}, we will give two examples to illustrate our main results. In  Section \ref{ques}, we will propose several open questions for future research. One interesting question among them is to determine those groups for which all Schur rings over them are symmetric algebras.

\section{Preliminaries\label{sect2}}
This section establishes notation, reviews basic definitions and terminology, and presents several preliminary lemmas required for the subsequent proofs.
\subsection{Notation and terminology}

Let $K$ be an arbitrary field with the identity $1_K$. In this paper, all groups are finite and all modules are finite-dimensional right modules unless stated otherwise. Let $A$ be an $K$-algebra. The category of all finite-dimensional right $A$-modules is denoted by $A\modcat$. For any group of algebra automorphisms of $A$, the subalgebra of $H$-fixed points in $A$ is denoted by $A^H$.

\medskip
Let $\Omega$ be a finite set and $G$ a finite group. The symbols $|\Omega|$ and $|G|$ denote the cardinality of $\Omega$ and the order of the group $G$, respectively. For a prime $p$, the group $G$ is called a {\it $p$-group} (respectively, {\it $p'$-group}) if $|G|$ is a $p$-power (respectively, $p\nmid |G|$). By $H<G$ and $H\leq G$ we mean that $H$ is a proper subgroup and a subgroup of $G$, respectively, and we use $|G:H|$ to denote the index of $H$ in $G$. For a finite group $G$ and a subgroup $H$, we use $K_H$ to denote the trivial $KH$-module. For a $KH$-module $M$, we write $M^G_H := M \otimes_{KH} KG$ for the induced module. For a $KG$-module $N$, we use $N_H$ to mean the restriction $N$ as a $KH$-module. For two $KG$-modules $M$ and $N$, by $M \mid N$ we mean $M$ is isomorphic to a direct summand of $N$ (as $KG$-modules). For Green's theory on modular representation of finite groups, for instance, vertices of indecomposable modules and Green correspondence, we refer the reader to \cite{Al}.

\medskip
Let $\Omega := \{1, 2, \ldots, n\}$ be a finite set. We denote by $A_n$ and $S_n$ the alternating and symmetric groups on $\Omega$, respectively. Throughout, we denote by $1_\Omega$ the diagonal subset of the Cartesian product $\Omega \times \Omega$, given by $\{(\alpha, \alpha) \mid \alpha \in \Omega\}$. For a subset $\Delta \subseteq \Omega$, we denote by $1_\Delta$ the diagonal subset of $\Delta \times \Delta$ consisting of $\{(\alpha, \alpha) \mid \alpha \in \Delta\}$.
For a subset $r \subseteq \Omega \times \Omega$ and an element $\gamma \in \Omega$, we define
$$r^* := \{(\beta, \alpha) \mid (\alpha, \beta) \in r\}, \quad \gamma r := \{\delta \in \Omega \mid (\gamma, \delta) \in r\}, \quad \mathrm{{and}} \quad \Omega(r) := \{\alpha \in \Omega \mid \alpha r \neq \emptyset\}.$$
Let $ G \leq \mathrm{Sym}(\Omega) $ be a transitive permutation group, where we denote by $\mathrm{Sym}(\Omega) $ the symmetric group on a nonempty set $ \Omega $. We denote the point stabilizer of a point $ \alpha \in \Omega $ by $ G_\alpha $, and we call the cardinalities of $G_\alpha$-orbits the \emph{subdegrees} of $ G $.
The group $ G $ is \emph{primitive} if every point stabilizer of $ G $ is a maximal subgroup of $G $. The $K $-linear span of $ \Omega $, denoted $ K\Omega $, becomes a $ KG $-module (called the \emph{natural} $KG $-module) in the following way: 
\[
\big( \sum_{\alpha \in \Omega} a_\alpha \alpha \big) \cdot g := \sum_{\alpha \in \Omega} a_\alpha (\alpha \cdot g), \quad, \forall~ g \in G, ~\forall~ \sum_{\alpha \in \Omega} a_\alpha \alpha \in K\Omega.
\]
For any non-empty subset $ X \subseteq G $, we define the \emph{sum element} of $X$ by $ \underline{X} = \sum_{x \in X} x $.
Fix a point $ \alpha \in \Omega $, and let $ H = G_\alpha $ be its point stabilizer. Then it is directly to verify that $ K\Omega \cong \underline{H}KG $ as
$KG$-modules.
Further, let $ p $ denote the characteristic of $ K $. If $ p \nmid |G:H| $, then $ \underline{H}KG $ admits the decomposition
\[
\underline{H}KG = \underline{G}KG \oplus \left( \underline{H} - [G:H]^{-1}\underline{G} \right) KG.
\]
This decomposition shows that $K_G \cong \underline{G}KG $ is a direct summand of $ K\Omega $ whenever $ p \nmid |G| $. This fact will be used throughout this paper without any reference.

\medskip

For a positive integer $ n $, we denote by $ M_n(K) $ the full $ n \times n $ matrix algebra over $ K $; we further denote by $P_n(K) $ the set of permutation matrices in $ M_n(K) $. For $ 1 \leq i, j \leq n $, the matrix unit $ e_{ij} \in M_n(K) $ is the matrix with a 1 in the $ (i,j) $-entry and 0s elsewhere. For a permutation $\sigma \in S_n $, the associated permutation matrix is given by $ \sum_{j=1}^n e_{j, \sigma(j)} $.

\subsection{Chevalley groups and groups with a split $BN$-pair}
Let $\mathfrak{g}$ be a finite dimensional complex simple Lie algebra with root system $\Phi$. Then Chevalley in \cite{Che} found that $\mathfrak{g}$ has (at least) a good basis with respect to a given Cartan decomposition of $\mathfrak{g}$, called the {\it Chevalley basis}. The $\mathbb{Z}$-span of this basis is a Lie subalgebra of $\mathfrak{g}$, say $\mathfrak{g}_0$. Given any field $K$, Chevalley in \cite{Che} constructed a group in the automorphism group of the Lie algebra $\mathfrak{g}_0\otimes_{\mathbb{Z}} K$, called the {\it Chevalley group} associated with $(\mathfrak{g},K)$. Let $\mathrm{GF}(q)$ be a finite field of $q$ elements. We abbreviate $G(\mathfrak{g},\mathrm{GF}(q))$ by $G(\mathfrak{g},q)$. The Chevalley groups $G(\mathfrak{g},q)$ for a prime power are the finite simple groups of Lie-type, with very few exceptions. They include four families of linear simple groups: $\mathrm{PSL}(n,q)$ (the projective special linear group), $\mathrm{PSU}(n,q)$ (the projective special unitary group), $\mathrm{PSp}(n,q)$ (the projective symplectic group), and $\mathrm{P\Omega^\epsilon}(n,q)$ where $\epsilon=\pm 1$. The following table lists exceptional (untwisted) Chevalley groups.

\begin{table}[htbp]
    \centering
    \begin{tabular}{ccc}
        \hline
        \textbf{Group} & \textbf{Rank} & \textbf{Comment} \\
        \hline
        $G_2(q)$ & 2 & $q=3^{2n+1}$ \\
        $F_4(q)$ & 4 & $q=2^{2n+1}$ \\
        $E_6(q)$ & 6 & $q=2^{2n+1}$ \\
        $E_7(q)$ & 7 &  \\
        $E_8(q)$ & 8 & \\
        \hline
    \end{tabular}
\end{table}

The concept of a $BN$-pair is originated to Tits in \cite{Tits}. They are useful not merely for the study of Chevalley groups, but also in connection with other types of classical simple groups.

\begin{defi}\label{BN}
A pair of subgroups $B, N$ of a group $G$ is called a $BN$-pair at characteristic $p$ if the following axioms are satisfied:
\vskip 2mm
$BN.1$ $G=\langle B, N\rangle$.

$BN.2$ $H=B\cap N$ is a normal subgroup of $N$.

$BN.3$ $W=N/H$ is generated by involutions $w_i$ for $i\in I$, and a preimage of $w_i$ is denoted by $\overline{w_i}$.

$BN.4$ For any $n\in N$ and any $i\in I$, $B\overline{w_i}B \cdot BnB\subseteq B\overline{w_i}nB\cup BnB$ and $B\overline{w_i}B\neq B$.

\medskip
The $BN$-pair is called a {\it split $BN$-pair at characteristic $p$} if the following additional conditions are satisfied:

\medskip
$BN.5$ $H$ is an abelian $p'$-group and $B$ has a normal $p$-subgroup $U$ such that $B=U\rtimes H$.

$BN.6$ $\cap_{n\in N} B^n=H$.
\end{defi}
Let $G$ be a finite group with a split $BN$-pair at characteristic $p$. The group $W$ in $BN.3$ is usually called the {\it Weyl group} of $G$, and we denote it by $\mathcal{W}_G$. A conjugate of $B$ is called a {\it Borel subgroup} of $G$, and any subgroup containing a Borel subgroup is called a {\it parabolic subgroup} of $G$. A classical example for finite groups with a split $BN$-pair is the general linear group over a finite field, where its subgroups of upper triangle matrices and monomial matrices (i.e., matrices with only non-zero entry in each row and column) form a split $BN$-pair. It is well known that any finite Chevalley group $G(\mathfrak{g}, p^n)$ has a split $BN$-pair at characteristic $p$ (see for instance \cite[Chapter 8]{Carter}.

\medskip
The following result was first proved by N. Tinberg in \cite{Ti}.

\begin{Lem}\label{Fro}
Let $G$ be a finite group with a split $BN$-pair at characteristic $p$ and let $K$ be a field of characteristic $p$. Then $\End_{KG}(K(G/U))$ is a Frobenius algebra.
\end{Lem}

\subsection{Hecke algebras}
The study of Hecke algebras originates in E. Hecke's 1937 work on endomorphisms of spaces of elliptic modular forms \cite{Hecke}. The notion of abstract Hecke algebras and Hecke pairs was later introduced by G. Shimura in 1959 \cite{SH}. Hecke algebras are closely connected to the representation theory of classical groups, and much of the existing literature focuses on Hecke algebras over fields of characteristic zero and their deformations. In this paper, however, we introduce and investigate abstract Hecke algebras in a setting suited to our purposes.
\begin{defi}
Let $G$ be a group (not necessarily finite) with a subgroup $H$. The pair $(G,H)$ is called a {\it Hecke pair} if the double coset $HgH$ for any $g\in G$ decomposes as finitely many left cosets of $H$ and finitely many right cosets of $H$.
\end{defi}
Clearly, $(G,H)$ is always a Hecke pair whenever $H$ is a finite subgroup. Let $R$ be a commutative ring, and let $(G,H)$  be a Hecke pair. Let $T=\{t_1, t_2, \cdots, t_i, \cdots\}$ be a right transversal for $H$ in $G$ and let $T_0=\{t_{j_1}, t_{j_2},\cdots, t_{j_k}, \cdots\}$ be a set of representative elements of $(H,H)$-double cosets in $G$. Set $\mathcal{H}_R(G,H)={\rm span}_R\{Ht_{j_k}H\mid t_{j_k}\in T_0\}$, the free $R$-module with all $(H,H)$-double cosets in $G$ as a basis.
\medskip

On the other hand, the $R$-linear span of all right cosets of $H$ in $G$ is naturally an $RG$-module. Any element $\varphi$ in the endomorphism ring $\End_{RG}(R(G/H))$ is uniquely determined by its action on the coset $H$ as $(Hg)\varphi=(H)\varphi\cdot g$ for all $g\in G$. In particular, $(H)\varphi \cdot h=(H)\varphi$ for all $h\in H$. For any $t_{j_k}\in T_0$, let $\varphi_{j_k}$ be the unique element in $\End_{RG}(R(G/H))$ such that $(H)\varphi_{j_k}=\sum_{t_i\in T\cap Ht_{j_k}H}Ht_i$. Then it is clear that the set of $\varphi_{j_k}$ with $t_{j_k}$ running over elements in $T_0$ is an $R$-basis of $\End_{RG}(R(G/H))$. Now, there is the following isomorphism of $R$-modules
$$\phi: \mathcal{H}_R(G,H)\rightarrow \End_{RG}(R(G/H)), \sum_{t_{j_k}\in T_0}a_{j_k}Ht_{j_k}H\mapsto \sum_{t_{j_k}\in T_0}a_{j_k} \varphi_{j_k}.$$
\begin{defi}
The free $R$-module $\mathcal{H}_R(G,H)$ with the multiplication given by $x\cdot y:=((x)\phi\cdot (y)\phi)\phi^{-1}$ for all $x, y\in \mathcal{H}_R(G,H)$ is an $R$-algebra, called the {\it Hecke algebra} over $R$ associated to the Hecke pair $(G,H)$.
\end{defi}
Let $H$ be a subgroup of a group $G$. The core of $H$ in $G$, denoted $H_G$, is defined as the intersection of all its conjugates $H_G = \bigcap_{g \in G} g^{-1}Hg$. Clearly, $H_G$ is a normal subgroup of $G$. The following statement is then immediate.

\begin{Lem}\label{modulo}
Let $(G,H)$ be a Hecke pair. Then $(G/H_G, H/H_G)$ is also a Hecke pair and the Hecke algebras $\mathcal{H}_R(G,H)$ and $\mathcal{H}_R(G/H_G,H/H_G)$ are isomorphic over any commutative rings $R$.
\end{Lem}

The right multiplication of $G/H_G$ on the right cosets of $H/H_G$ is core-free and therefore $G/H_G$ may be viewed as a transitive permutation group on the right cosets of $H/H_G$. This establishes a relation between Hecke algebras and endomorphism algebras of permutation modules for transitive permutation groups.
\medskip

The following result was first proved by Iwahori in \cite[Theorem 3.2]{Iwa}.

\begin{Lem}\label{C-H}
Let $G=G(\mathfrak{g},q)$ be a finite Chevalley group with a Borel subgroup $B$. Suppose $K$ is a field for which $q\cdot 1_K$ is an invertible element. Then $\mathcal{H}_K(G,B)\cong K\mathcal{W}_G$ as algebras.
\end{Lem}
This result can be generalized to arbitrary finite group $G$ with a split $BN$-pair in place of a Chevalley group.
\begin{Lem}\label{He}
Let $G$ be a finite groups with a split $BN$-pair at characteristic $q$. Suppose $K$ is a field for which $q\cdot 1_K$ is an invertible element. Then $\mathcal{H}_K(G,B)\cong K\mathcal{W}_G$ as algebras.
\end{Lem}
\subsection{Coherent algebras and Schur rings}
\subsubsection{Coherent algebras}

In this subsection, we introduce the notion of coherent algebras.
\medskip

To introduce coherent algebras, we need the concept of coherent configurations. A {\it coherent configuration} is defined as a pair $ \mathcal{C}=(\Omega, S) $, where $S $ partitions $\Omega \times \Omega $, subject to the following conditions:

(1) $1_{\Omega} \in S^{\cup}$, where the elements of $S^{\cup}$ are unions of some elements from $S$,

(2)  $S^{*} = S$, where $S^{*}:=\{r^{*}\mid r\in S\}$,

(3)  For any $r, s, t\in S$, the number $C^t_{rs} := |\alpha r\cap \beta s^{*}|$ does not depend on the choice of $(\alpha, \beta)\in t$, here $|\alpha r\cap \beta s^{*}|$ is the number of $\gamma\in \Omega$ such that $(\alpha,\gamma)\in r, ~(\gamma, \beta)\in s$.

\medskip
Each element in the set $S$ is called a {\it basic set}. For any $\delta \in \Omega ({s})$, the positive integer $|\delta s|$ equals to the intersection number $c_{s s^{*}}^{1_{\Omega ({s})}}$, hence does not depend on the choice of $\delta$ and we call it the {\it valency} of $s$. By \cite{CP}, the following equalities hold:
$$
|t| C_{r s}^{t^{*}}=|r| C_{s t}^{r^{*}}=|s| C_{t r}^{s^{*}}, \quad r, s, t \in S.
$$

For a permutation group $G \leq \operatorname{Sym}(\Omega)$, one can associate a coherent configuration $\operatorname{Inv}(G,\Omega)$ to it, as described in \cite[Chapter 2]{CP}. The set of its basic relations consists precisely of all 2-orbits of the action of $G$ on $\Omega \times \Omega$. A relation $s \in \operatorname{Orb}(G, \Omega^2)$ is called \emph{reflexive} if $(\alpha, \alpha) \in s$ for some (hence all) $\alpha \in \Omega$.

\medskip
Now let $A_{\mathbb{Z}}(\mathcal{C}) := \mathbb{Z} S$ be the free $\mathbb{Z}$-module with elements in $S$ as a basis. It becomes a ring by defining multiplication on basis elements as
\[
r \cdot s := \sum_{t \in S} c^t_{rs} \, t \qquad (r, s \in S),
\]
and extending this linearly to all of $A_{\mathbb{Z}}(\mathcal{C})$. The algebra $A_R(\mathcal{C})=A_{\mathbb{Z}}(\mathcal{C})\otimes_{\mathbb{Z}} R$ for a commutative ring $R$ is called the \emph{$R$-coherent algebra} associated with $\mathcal{C}$. $\mathbb{C}$-Coherent algebras, was first introduced by D. Higman \cite{DH}, in his attempt to study quasisymmetric designs.

\medskip

The following result illustrates the importance of $K$-coherent algebras in the study of endomorphism algebras of permutation modules for permutation groups.

\begin{Lem}\label{Coherent}
Let $G$ be a permutation group (not necessarily transitive) on a set $\Omega$. Then for any commutative ring $R$, there is an isomorphism $\End_{RG}(R\Omega)\cong A_R(\mathcal{C})$, where $\mathcal{C}=(\Omega, S)$ and $S$ is the set of the orbits of the induced diagonal action of $G$ on $\Omega\times \Omega$.
\end{Lem}
Without loss of generality, we may take $\Omega$ to be the set $\{1, 2, \cdots, n\}$. Then $K$-coherent algebra $A_K(\mathcal{X})$ can be realized as subalgebras of the matrix algebra $M_n(K)$ by the following way:
For each $X\in S$, define $A_X\in M_n(K)$ to be the matrix such that
\[
(A_X)_{ij}=\left\{
\begin{array}{rcl}
1   &    &{(i,j)\in X}\\
0   &    &{(i,j)\notin X}\\
\end{array}
\right.
\]
Then the linear span $\sum_{X\in S} KA_X$ of the matrices $A_X$ for $X\in S$ is indeed a subalgebra of $M_n(K)$, and it is directly to check that $A_K(\mathcal{C})\cong \sum_{X\in S} KA_X$. In the rest of this paper, we also call $\sum_{X\in S} KA_X$ a $K$-coherent algebra.
\subsubsection{Schur rings}

In the literature, since Schur algebras are usually referred to endomorphism algebras for a specific module over the symmetric groups, we abuse notation and use the notion of Schur rings (even they are $K$-algebras in our sense) in this paper. Schur rings are usually studied as algebras over the complex number field in relation to permutation groups and combinatorics. But for the purpose of this paper, we will study Schur rings over arbitrary field.

\medskip
Given a finite group $G$, a subring $\mathcal{A}$ of the algebra $KG$ is called a {\it $K$-Schur ring}  over $G$ if it has a linear $K$-basis consisting of elements $X$, where $X$ runs over a family $\mathcal{S}(\mathcal{A})$ of pairwise disjoint nonempty subsets of $G$ such that

$(1)$ $\{e\} \in \mathcal{S}$,

$(2)$ $\bigcup\limits_{X\in \mathcal{S}}X=G$,

$(3)$ $X^{-1}:=\{x^{-1} \mid x\in X\}\in \mathcal{S}$ for any $X\in \mathcal{S}(\mathcal{A})$.

\medskip
The elements of $\mathcal{S}(\mathcal{A})$ are called the \emph{basic sets} of $\mathcal{A}$. These basic sets form a partition of $G$, called the {\it Schur partition } of $G$. Clearly, if the field $K$ is fixed, then the Schur ring is uniquely determined by the Schur partition associated to it. The group algebra $KG$ and the algebras spanned by $1,~ \sum_{1\neq g\in G}g$ are trivial examples of $K$-Schur rings over $G$, which are called {\it trivial Schur rings} over $G$. A subgroup $H$ is called an {\it $\mathcal{A}$-subgroup} if it is a union of basic sets. Clearly, the identity subgroup and the group $G$ itself are always $\mathcal{A}$-subgroups. Let $H$ be a normal $\mathcal{A}$-subgroup. Then it is directly to check that either $HX=HY$ or $HX\cap HY=\emptyset$ for any two basic sets $X, Y\in \mathcal{S}(\mathcal{A})$.  Consequently, the linear span of $HX$ for all $X\in \mathcal{S}(\mathcal{A})$ is a Schur ring over $G/H$, which is denoted by $\mathcal{A}_{G/H}$.

\medskip

Given a $K$-Schur ring $\mathcal{A}$, define $X'=\{(g, gx)\in (G, G)\mid x\in X, g\in G)\}$ for all $X\in \mathcal{S}(\mathcal{A})$. Then elements in $Z_{\mathcal{A}}:=\{X'\mid X\in \mathcal{S}(\mathcal{A})\}$ form a partition of $G\times G$, and it is directly to check that $(G, Z_{\mathcal{A}})$ is a coherent configuration. We say this coherent configuration is the one associated with the Schur partition of $\mathcal{A}$. With this notation, the following result is obvious.
\begin{Lem}\label{CO-SCH}
Any $K$-Schur ring is isomorphic to the $K$-coherent algebra associated with the coherent configuration obtained by its Schur partition.
\end{Lem}
Though $K$-Schur rings may be realized as $K$-coherent algebras, but they are of interest in their own right as they are closely related to representation of finite groups.

\medskip
Let $\mathcal{A}$ be a $K$-Schur ring over $G$. If there exists a proper $\mathcal{A}$-subgroup $H<G$ and a normal $\mathcal{A}$-subgroup $U>1$ of $G$ contained in $H$ such that any basic set $X\subseteq G\setminus H$ is a union of $U$-cosets, then we say $\mathcal{A}$ is a wedge product of $\mathcal{A}_H$ and $\mathcal{A}_{G/U}$, and we denote $\mathcal{A}$ by $\mathcal{A}_H\wr_{H/U}\mathcal{A}_{G/U}$. In case $U=H$, we simply write $\mathcal{A}_H\wr \mathcal{A}_{G/H}$.

\medskip
Recall that we define $\underline{X}=\sum_{x\in X} x$ for any $X\subseteq  G$.
Then it is known from \cite[Theorem 1.6]{I} that for $T \leq \Aut(G)$ the $K$-linear space $\mbox{span}\{\underline{X}\mid X\in {\rm Orb}(T, G)\}$ is a $K$-Schur ring over $G$. This Schur ring is called {\it a cyclotomic $K$-Schur ring} over $G$. Moreover, $\mathcal{S}(\mathcal{A})={\rm Orb}(T, G)$.
\medskip

Given two Schur rings $\mathcal{A}_H$ and $\mathcal{A}_G$ over groups $H$ and $G$, respectively, the $K$-linear space
$$
\mbox{span}\{(\underline{X}, \underline{Y})\mid X\in \mathcal{S}(\mathcal{A}_H), Y\in \mathcal{S}(\mathcal{A}_G)\}
$$is clearly also a $K$-Schur ring over the group $H\times G$. This Schur ring is called the {\it tensor product} of $\mathcal{A}_H$ and $\mathcal{A}_G$, and it is denoted by $\mathcal{A}_H\otimes\mathcal{A}_G$.
\medskip

Now let $G$ be a transitive permutation group on a set $\Omega$ that contains a regular subgroup $N$. Fix a point $\alpha \in \Omega$; then $G = G_\alpha N$.
Define a bijection $\pi : \Omega \to N$ by letting $(\beta)\pi$ be the unique element of $N$ satisfying $\alpha \cdot (\beta)\pi = \beta$ for every $\beta \in \Omega$. Using $\pi$, we can transfer the action of $G$ to $N$ via
\[
(\beta)\pi \cdot g \: := \: (\beta \cdot g)\pi \qquad (\beta \in \Omega,\; g \in G),
\]
which makes \(N\) a \(G\)-set and identifies \(N_{\text{right}}\) (the right regular permutation representation of $N$ on itself) as a subgroup of \(G\). The following fact is well known.

\begin{Lem}\label{regul}
Let $G$ be a finite transitive group on a set $\Omega$ with a regular subgroup $N$. Then the orbits in ${\rm Orb}(G_{\{1\}}, N)$ form a Schur partition of $N$ and $\End_{KG}(K\Omega)$ is isomorphic to the $K$-Schur ring $\mathcal{A}_H$ defined by this Schur partition.
\end{Lem}
 By Lemma \ref{regul}, if a finite group $G$ acts transitively on a set $\Omega$ and contains a regular subgroup $N$, then there is an isomorphism between the Hecke algebra $\mathcal{H}_K(G, G_\alpha)$ and the Schur ring $\mathcal{A}_H$, where $G_\alpha$ is a point stabilizer for some $\alpha\in \Omega$ in $G$.

\subsection{Symmetric algebra, Frobenius algebras and self-injective algebras}

Let $A$ be an $K$-algebra. The algebra $A$ is defined to be {\it symmetric} if ${\rm Hom}_{K}(A, K) \cong A$ as $A$-$A$-bimodules. Equivalently, there exists a non-degenerate symmetric associative bilinear form $f: A \times A \rightarrow K$, that is, a mapping $f: A \times A \rightarrow K$ satisfying the following conditions:
\begin{enumerate}
\item[\rm (i)] $ f(a + b, c) = f(a, c) + f(b, c) $ and $ f(a, b + c) = f(a, b) + f(a, c) $ for all $ a, b, c \in A $;
\item[\rm (ii)] $ f(ab, c) = f(a, bc) $ for all $ a, b, c \in A $;
\item[\rm (iii)] $ f(a, b) = f(b, a) $ for all $ a, b \in A $;
\item[\rm (iv)] $ f(ta, b) = f(a, tb) = t f(a, b) $ for all $t \in K $ and $ a, b \in A $;
\item[\rm (v)] If $ f(a, b) = 0 $ for all $ b \in A $, then $ a = 0 $.
\end{enumerate}
If $A$ is a symmetric algebra, then $eAe$ is also a symmetric algebra for any idempotent $e\in A$.

\medskip
The algebra $A$ is called \emph{Frobenius} if $f$ satisfies all these conditions except (iii). Clearly, symmetric algebras are always Frobenius algebras, but the converse is not true in general. There are many nice properties of Frobenius algebras; we refer readers to \cite{SY} for a systematic introduction.

\medskip
The algebra $A$ is called {\it self-injective} if the regular module $A$ is an injective module, or equivalently, all projective $A$-modules are injective, or equivalently, the ${\rm Ext}$-functor $\Ext^1_A(-,A)$ vanishes for all $A$-modules. It is well-known that a Frobenius algebra is necessarily a self-injective algebra. And the following fact is well-known (see for instance \cite[Proposition 4.6, p. 348]{SY}):
$$\mbox{All commutative self-injective algebras are symmetric algebras.}$$

\subsection{Preliminary results}
In this section, we present some propositions that will be useful in our subsequent proofs.  The classification of Schur rings over finite cyclic groups was established by Leung and Man \cite{KHL1, KHL2}. Their work leads directly to the following result.

\begin{Prop}\label{C-S}
Let $K$ be a field and let $\mathcal{A}$ be a $K$-Schur ring over a cyclic group $G$. Then one of the following holds:
\begin{enumerate}
 \item[\rm (i)] $\mathcal{A}$ is trivial;
 \item[\rm (ii)] $\mathcal{A}$ is a cyclotomic $K$-Schur ring;
 \item[\rm (iii)] there exist $\mathcal{A}$-subgroups $H$ and $U$ such that $G=H\times U$ and $\mathcal{A}=\mathcal{A}_H\times\mathcal{A}_U$ is a tensor product of two $K$-Schur rings $\mathcal{A}_H$ and $\mathcal{A}_U$;
 \item[\rm (iv)] there exist proper $\mathcal{A}$-subgroups $U\leq H\leq G$ such that $\mathcal{A}=\mathcal{A}_H\wr_{H/U}\mathcal{A}_{G/U}$ is a wedge product of two $K$-Schur rings $\mathcal{A}_H$ and $\mathcal{A}_{G/U}$.
 \end{enumerate}
\end{Prop}

\begin{Prop}\label{Sym}

Let $K$ be a field and $G$ be a half-transitive permutation group on a set $\Omega$. Suppose that, for every $\alpha\in\Omega$, the sizes of $G_{\alpha}$-orbits on $\Omega$ are invertible in the field $K$. Then the endomorphism algebra ${\rm End}_{KG}(K\Omega)$ is a symmetric algebra.
\end{Prop}

{\bf Proof.} For convenience, we take $\Omega$ to be the set $\{1,2,\cdots,n\}$. Then $G\leq S_n$. Denote the orbits of the diagonal action of $G$ on $\Omega\times\Omega$ by $S = \{s_1,\cdots,s_t,s_{t + 1},\cdots,s_m\}$, where $\{s_1,\cdots,s_t\}$ are exactly the reflexive orbits. Then $\mathcal{C}=(G,S)$ is a coherent configuration and by Lemma \ref{Coherent} we have isomorphism of algebras $$\End_{KG}(K\Omega)\cong A_K(\mathcal{C})=\mathop  \oplus \limits_{i =1}^m KA_{s_i}.$$To complete the proof, we need to show that $A_K(\mathcal{C})$ is a symmetric algebra. We fix $\mathcal{B}=\{A_{s_i},1\leq i\leq m\}$ as a basis of $A_K(\mathcal{C})$. For $s_k\in S$, we conclude from \cite[Section 1.1]{CP} that $\Omega(s_k)$ is an $G$-orbit. For all $A_{s_k},A_{s_{l}}\in \mathcal{B}$, we define a mapping as follows:
$$\varphi: A_K(\mathcal{C})\times A_K(\mathcal{C})\rightarrow K,~(A_{s_k}, A_{s_{l}})\mapsto C^{1_{\Omega(s_k)}}_{s_k s_{l}}1_K.$$

According to \cite[Chapter 1]{CP}, $\varphi$ is clearly bilinear, and $\varphi(A_{s_k}, A_{s_{l}}A_{s_{r}})$ and $\varphi(A_{s_k}A_{s_l}, A_{s_r})$ are just the coefficient of $A_{1_{\Omega(s_k)}}$ in the expression of $A_{s_k} A_{s_l}A_{s_r}$ as a linear combination of elements in the basis $\mathcal{B}$. Thus $\varphi(A_{s_k}, A_{s_l}A_{s_r})$ and $\varphi(A_{s_k}A_{s_l}, A_{s_r})$ are equal, and therefore $\varphi$ is associative. Now it remains to show that $\varphi$ is symmetric and non-degenerate.

To show that $\varphi$ is symmetric, it suffices to show that $C^{1_{\Omega(s_k)}}_{s_k s_l}1_K=C^{1_{\Omega(s_l)}}_{s_l s_k}1_K$ for all $s_k,s_l\in S$. Since $G$ is half-transitive on $\Omega$, that is, all $G$-orbits have the same size, it follows from \cite[Chapter 1]{CP} that $$C^{1_{\Omega(s_k)}}_{s_k s_l}|\Omega(s_k)|1_K= C^{1_{\Omega(s_l)}}_{s_l s_k}|\Omega(s_l)|1_K,$$ whence we obtain the desired equality $C^{1_{\Omega(s_k)}}_{s_k s_l}1_K=C^{1_{\Omega(s_l)}}_{s_l s_k}1_K$.

Finally, we show that $\varphi$ is non-degenerate, that is, we need to show that $0$ is the only element in
$$\rad(\varphi):=\{x\in A_K(\mathcal{C})\mid \varphi(x,y)=0~ \mbox{for~all} ~y\in A_K(\mathcal{X})\}.$$
Let $x\in \rad(\varphi)$ be an arbitrary element. As $A_K(\mathcal{C})=\mathop  \oplus \limits_{i =1}^m KA_{s_i}$, we can write $x=\sum^m_{i=1} a_{s_i} A_{s_i}$, where $a_{s_i}\in K$ for $1\leq i\leq m$. For any $s_k, s_l\in S$, it is clear that $C^{1_{\Omega(s_k)}}_{s_k s_l}= 0$ if $s_k\neq s^{*}_l$. Therefore, we obtain
$0=\varphi(x, s^{*}_k)=a_{s_k} C^{1_{\Omega(s_k)}}_{s_k s^{*}_k}$ for $1\leq k\leq m$.
Let $(\alpha,\beta)\in s_k$. Then $C^{1_{\Omega(s_k)}}_{s_k s^{*}_k}=|\beta^{G_\alpha}|$,
where $\beta^{G_\alpha}=\{\beta\cdot x\mid x\in G_\alpha\}$. By our assumption on $G_\alpha$, then $|\beta^{G_\alpha}|$ is invertible in $K$. This implies that $a_{s_k}=0$ for $1\leq k\leq m$, whence $x=0$. This completes the proof. \qed

\medskip
As a consequence of Proposition \ref{Sym}, we have the following result.
\begin{Prop}\label{cop}
Let $G$ be a transitive permutation group acting on a set $\Omega$. Suppose that $G$ contains an abelian regular subgroup $T$, and all the subdegrees of $G$ are relatively prime to the order of $T$. Then $G$ is an $S$-permutation group.
\end{Prop}

{\bf Proof.} Let $K$ be an arbitrary field of characteristic $p$. We shall show that $\End_{KG}(K\Omega)$ is a symmetric algebra. By Proposition \ref{Sym}, we need only consider the case where $p$ divides some subdegree of $G$.

Suppose that $p$ divides some subdegree of $G$. By our assumption, $\gcd(p,|T|)=1$. Lemma \ref{regul} implies the existence of a $K$-Schur ring $\mathcal{A}$ over $T$ such that $\End_{KG}(K\Omega)\cong \mathcal{A}$ as algebras. Since $p$ does not divide $|T|$, Maschke's theorem implies that $KT$ is semisimple. In particular, $KT$ has no non-zero nilpotent ideal, and hence no non-zero nilpotent central element. Because $T$ is abelian, the group algebra $KT$ is commutative; consequently, every element of $KT$ is central. Thus $KT$ contains no non-zero nilpotent elements. It follows that the subalgebra $\mathcal{A}$ of $KT$ also has no non-zero nilpotent elements, and therefore $\mathcal{A}$ is semisimple. Hence $\End_{KG}(K\Omega)\cong \mathcal{A}$ is a semisimple algebra, and thus a symmetric algebra, as required.
  \qed

\begin{Prop}\label{rank}
Let $K$ be a field and $G$ a permutation group on a set $\Omega$. If $G$ has a transitive subgroup $H$ such that $\rank(H)=\rank(G)$, then $\End_{KG}(K\Omega)\cong \End_{KH}(K\Omega)$ as algebras.
\end{Prop}
{\bf Proof.}
Let $n$ denote the cardinality of $\Omega$, and let $S_1$ (respectively, $S_2$) be the orbits of the diagonal action of $G$ (respectively, $H$) on the Cartesian product $\Omega\times \Omega$. Then $\mathcal{C}_1:=(\Omega, S_1)$ and $\mathcal{C}_2:=(\Omega, S_2)$ are coherent configurations, and by Lemma \ref{Coherent} we have $\End_{KG}(K\Omega) \cong A_K(\mathcal{C}_1)$ and $\End_{KH}(K\Omega) \cong A_K(\mathcal{C}_2)$ as algebras. Since $H \leq G$, we derive that any orbit in $S_1$ is a union of orbits in $S_2$. Note that $|S_1| = \operatorname{rank}(G)$ and $|S_2| = \operatorname{rank}(H)$. Thus it follows from the assumption $\operatorname{rank}(H) = \operatorname{rank}(G)$ that $|S_1|=|S_2|$.
This, together with the inclusion $A_K(\mathcal{C}_1)\subseteq A_K(\mathcal{C}_2)$ implies $A_K(\mathcal{C}_1)= A_K(\mathcal{C}_2)$. Consequently,
$
\End_{KG}(K\Omega) \cong A_K(\mathcal{C}_1)= A_K(\mathcal{C}_2) \cong \End_{KH}(K\Omega)
$.

 \qed

\begin{Prop}\label{subdegree}
Let $G$ be a $\frac{3}{2}$-transitive permutation group acting on the set $\Omega$, and let $H:=G_\alpha$ for some $\alpha\in \Omega$. Then, for any prime number $p$ that divides the non-trivial subdegrees of $G$, the normalizer of a Sylow $p$-subgroup of $H$ is contained in $H$.
\end{Prop}
{\bf Proof.} Since $G$ is $\frac{3}{2}$-transitive, all of the non-trivial subdegrees of $G$ are identical. Let $p$ be a prime dividing the non-trivial subdegrees of $G$ and let $P$ be a Sylow $p$-subgroup of $H$. If $N_G(P)$ is not contained in $H$, then there exists an element $x\in N_G(P)\setminus H$ such that $P\leq H\cap H^x$. Consequently, $p$ does not divide $|H:H\cap H^x|$. However, $|H:H\cap H^x|=|\{\alpha\cdot xh\mid h\in H\}|$ is a subdegree of $G$. This is a contradiction to our assumption. \qed
\vskip 2mm
Combining Lemmas \ref{Fro} and \ref{C-H}, we can prove the following.
\begin{Prop}\label{He-F}
Let $G$ be a finite group with a split $BN$-pair at characteristic $q$. Then $\mathcal{H}_K(G,B)$ is a Frobenius algebra over any field $K$.
\end{Prop}
\demo Let $p$ be the characteristic of $K$. Suppose $p\neq q$. By Lemma \ref{Fro}, $\mathcal{H}_K(G,B)\cong KW_G$ is a symmetric algebra, as group algebras over arbitrary fields are always symmetric algebras. Therefore, $\mathcal{H}_K(G,B)$ is a Frobenius algebra as symmetric algebras are Frobenius algebras. Thus in what follows we may assume that $p=q$.

Note that $B$ has an abelian $p'$-subgroup, say $H$, such that $B=U\rtimes H$. Let $x=\sum_{h\in H} UhU\in \mathcal{H}_{\mathbb{Z}}(G,B)$ be an element. As $\mathcal{H}_{\mathbb{Z}}(G,B)$ is a free $\mathbb{Z}$-module with $(U,U)$-double cosets as a basis, it follows that tensoring $\mathcal{H}_{\mathbb{Z}}(G,B)\otimes_{\mathbb{Z}}-$ with the following exact sequence $$0\ra \mathbb{Z}\ra \mathbb{Q}\ra \mathbb{Q}/\mathbb{Z}\ra 0$$remains an exact sequence. Thus $x$ may also be viewed as an element in $\mathcal{H}_{\mathbb{Q}}(G,B)$. In \cite[Remark 6]{Ca}, it was shown that $\frac{1_{\mathbb{Q}}}{|H|}\cdot x$ is a central idempotent in $\mathcal{H}_{\mathbb{Q}}(G,B)$. This together with the exactness of $\mathcal{H}_{\mathbb{Z}}(G,B)\otimes_{\mathbb{Z}}-$ implies that $x$ is central in $\mathcal{H}_{\mathbb{Z}}(G,B)$. Since $p=q$ and $H$ is an $p'$-group, it follows that $|H|$ is invertible in $K$. Thus $x\otimes_{\mathbb{Z}}\frac{1_K}{|H|}$ is a central idempotent in $\mathcal{H}_{\mathbb{Z}}(G,B)\otimes_{\mathbb{Z}} K=\mathcal{H}_K(G,B)$. Let $e=\frac{1_K}{|H|}\sum_{h\in H}h$. Then by \cite[Remark 9]{Ca} Cabanes pointed out that $K(G/B)\cong eK(G/U)$ is a direct summand of $K(G/U)$ as $KG$-module. Thus $\mathcal{H}_K(G,B)\cong e_0\mathcal{H}_K(G,U)e_0$ as algebras, where $e_0=\frac{1_K}{|H|}\cdot x$. Since $e_0$ is a central idempotent in $\mathcal{H}_K(G,U)$, it follows that $\mathcal{H}_K(G,B)$ is isomorphic to a sum of blocks of $\mathcal{H}_K(G,U)$. This together with  $\mathcal{H}_K(G,U)$ being a Frobenius algebra implies that $\mathcal{H}_K(G,B)$ is also a Frobenius algebra. \qed
\medskip

To end with this subsection, we show that $2$-transitive permutation groups and Frobenius permutation groups are $S$-permutation groups.

\begin{Prop}\label{2-tr}
Any $2$-transitive group is an $S$-permutation group.
\end{Prop}
{\bf Proof.}
 Let $ K $ be an arbitrary field of characteristic $p$, and let $G$ be a 2-transitive permutation group on a finite set $ \Omega $ of size $n $. Then the set of orbits $S$ for the diagonal action of $G$ on the Cartesian product $\Omega\times \Omega$ consists of two orbits, namely, $1_\Omega$ and $\Omega\times \Omega\setminus 1_\Omega$. To prove that $G$ is an $S$-permutation group, we need to show that $\End_{KG}(K\Omega) \cong A_K((\Omega, S))$ is always a symmetric algebra. Clearly, $A_K((\Omega, S)) = KI_n \oplus KJ ,$ where $J = \sum_{1 \leq i \neq j \leq n} e_{i,j} $ is the matrix with zeros on the diagonal and ones off the diagonal. Consequently,
\[
\operatorname{End}_{KG}(K\Omega) \cong
\begin{cases}
K[X]/(X^2), & \text{if } p \mid n, \\
K \times K, & \text{if } p \nmid n.
\end{cases}
\]
In both cases, $\operatorname{End}_{KG}(K\Omega) $ is a symmetric algebra. This completes the proof.
\qed

\medskip

Note that Frobenius permutation groups are necessarily $\frac{3}{2}$-transitive permutation groups.
\begin{Prop}\label{Frobenius}
Any Frobenius permutation group is an $S$-permutation group.
\end{Prop}
{\bf Proof.} Let $K$ be an arbitrary field of characteristic $p$ and $G$ a Frobenius permutation group on a set $\Omega$ with $H = G_{\alpha}$ the point stabilizer of some $\alpha\in \Omega$ in $G$. Then $H$ is a Frobenius complement in $G$. To show that $G$ is an $S$-permutation group, we need to show that $\End_{KG}(K\Omega)$ is always a symmetric algebra. By Lemma \ref{Sym}, we may assume that $p$ divides some subdegree of $G$. In particular, $p\mid |H|$. Since $G$ is a Frobenius group, by the definition of Frobenius groups, we see that $|H|$ and $|\Omega|$ are relatively prime. Thus, $p\nmid |\Omega|$.

Let $P$ be a Sylow $p$-subgroup of $H$. Then, by properties of Frobenius complements, we have $N_G(P) \leq H$. Since $K\Omega$ is isomorphic to $K^G_H$ as $KG$-modules and $P$ is a vertex of $K_H$, the trivial module $K_G$ is the Green correspondence of $K_H$, and the vertex of any indecomposable summand of $K^G_H$ other than $K_G$ is of the form $P\cap P^g$ for some $g\in G \setminus H$. However, on the other hand, since $G$ is a Frobenius group, $H\cap H^x = 1$ for all $x\in G \setminus H$. Thus, the vertices of all indecomposable summands of $K^G_H$, except for $K_G$, are the identity subgroup. This indicates that $K^G_H$ is a direct sum of $K_G$ and a projective module. Thus $\End_{KG}(K\Omega)\cong K\times \End_{KG}(M)$. Let $N$ be the sum of non-isomorphic indecomposable direct summands of $M$. Here, $N$ is referred to as the basic module of $M$. Then, according to the well-known Morita theorem, $\End_{KG}(M)$ and $\End_{KG}(N)$ are Morita equivalent. By the definition of $N$, we can find an idempotent $e$ of $KG$ such that $N$ is isomorphic to $eKG$ as $KG$-modules. Consequently, $\End_{KG}(N)$ is isomorphic to $eKGe$, which is known to be a symmetric algebra. We conclude that $\End_{KG}(M)$ is a symmetric algebra as it is Morita equivalent to $\End_{KG}(N)$. Hence, $\End_{KG}(K\Omega)\cong K\times \End_{KG}(M)$ is also a symmetric algebra. \qed
\section{Proof of Theorem \ref{main} \label{pf}}

The proof of Theorem \ref{main} is divided into two parts: one part deals with $p$-$S$-permutation groups and another part deals with $S$-permutation group.
\subsection{Criterion for $p$-$S$-permutation groups}

Recall that a transitive permutation group always has $1$ as a subdegree. Other subdegrees (may also be $1$) are called non-trivial subdegrees. In this subsection, we mainly prove the following result.

\begin{Prop}\label{p-S}
Let $G$ be a transitive permutation group on a set $\Omega$ of cardinality $n$. Then $G$ is a $p$-$S$-permutation group for a prime $p$ if one of the following is satisfied

$(i)$ all subdegrees of $G$ are coprime to $p$;

$(ii)$ $p\nmid \frac{|G|}{nm}$ for any non-trivial subdegree $m$ of $G$;

$(iii)$ $G$ has an abelian regular $p'$-subgroup;

$(iv)$ $n<2p$;

$(v)$ $p\nmid n$ and $G$ is of rank $3$;

$(vi)$ $\Omega=G(q)/\mathcal{P}$ and $G$ is the permutation group induced by the right multiplication of $G(q)$ on $\Omega$, where $G(q)$  is finite group with a split $BN$-pair at characteristic $q\neq p$ and $\mathcal{P}$ is a parabolic subgroup of $G(q)$ such that $p\nmid \frac{|\mathcal{P}|}{|B|}$.
\end{Prop}

\demo Let $K$ be a field of characteristic $p$, and let $H$ be the stabilizer in $G$ of a point $\alpha \in \Omega$. (i) This follows from Proposition \ref{Sym} together with the fact that every transitive permutation group is  half-transitive. (iii) This follows by an argument similar to that used in the proof of Proposition \ref{cop}.

Assume (ii). Then by assumption we have $H\cap H^g$ is always a $p'$-group for any $g\in G\setminus H$. With a similar argument as in the proof Proposition \ref{Frobenius}, we deduce that $\End_{KG}(K\Omega)$ is a symmetric algebra.

Assume (iv). If $n\leq p$, then all subdegrees of $G$ are coprime to $p$ and it follows from (i) that $G$ is always a $p$-$S$-permutation group. Thus we may assume $n\geq p+1$ and by (i) we may also assume that $G$ has a subdegree divisible by $p$. In particular, $p\mid |G|$. If $n=p+1$, then $G$ is necessarily a $2$-transitive permutation group and therefore an $S$-permutation group by Proposition \ref{2-tr}. Hence we assume $n\geq p+2$. Suppose $G$ is a primitive group. Then by \cite[Theorem 1.1]{LH}, $G$ is a $p$-$S$-permutation group. Thus we assume that $G$ is imprimitive. Then there exists a subgroup $M$ of $G$ such that $H<M<G$. The set $\Delta:=\{\alpha\cdot m\mid m\in M\}$ is then a block of $G$. Let $\mathcal{B}=\{\Delta\cdot g\mid g\in G\}$ be the block system containing $\Delta$. The induced action $G$ on $\mathcal{B}$ has kernel $N:=\cap_{g\in G}M^g$ and $G^{\mathcal{B}}\cong G/N$. By hypothesis, $|\mathcal{B}|<p$. Thus $G/N$ is a $p'$-group. This implies that $p\mid |N|$. Now, consider the set $\Gamma:=\{\alpha\cdot x\mid x\in N\}$. Then $\Gamma\subseteq \Delta$ and so $|\Gamma|\leq |\Delta|<p$. This implies that the index of $N_\alpha=N\cap H$ in $N$ is $|\Gamma|<p$. Thus the order $N/C$ is not divided by $p$, where $C:=\cap_{x\in N} N^x_\alpha$. Therefore, $p\mid |C|$ and so $p\nmid |G:C|$. Since $C\unlhd N\unlhd G$ and $p\nmid |G:C|$, it follows that $K^G_C$ is a semisimple $KG$-module. On the other hand, note that $C\leq H$ and $|H:C|$ is a divisor of $|G:C|$, it follows that $K_H\mid K^H_C$. Thus by transitivity of induction, $K\Omega\cong K^G_H\mid (K^H_C)^G_H=K^G_C$. This yields that $K\Omega$ is a semisimple $KG$-module. Thus $\End_{KG}(K\Omega)$ is a semisimple algebra and therefore is a symmetric algebra.

Assume (v). Then it follows from $p\nmid n$ that $K\Omega=K_G\oplus M$ decomposes as a direct sum of two modules such that $\Hom_{KG}(K_G,M)=\Hom_{KG}(M,K_G)=0$. Thus $\End_{KG}(K\Omega)\cong K\times \End_{KG}(M)$ as algebras. Since $G$ is of rank $3$, it follows that $\End_{KG}(K\Omega)$ is 3-dimensional and therefore $\End_{KG}(M)$ is 2-dimensional. Hence $\End_{KG}(M)$ is either semisimple or isomorphic to $K[x]/(x^2)$. In either cases, $\End_{KG}(M)$ is a symmetric algebra. This implies that $\End_{KG}(K\Omega)\cong K\times \End_{KG}(M)$ is a symmetric algebra.

Finally, we assume (vi). Let $B$ be a Borel subgroup contained in $\mathcal{P}$. Then it follows from the assumption $p\nmid |\mathcal{P}:B|$ that $K_{\mathcal{P}}\mid K^\mathcal{P}_B$. Thus, by transitivity of induction, $K^G_{\mathcal{P}}\mid K^G_B$. This implies that there exists an idempotent $e\in \End_{KG}(K(G/B))$ such that $\End_{KG}(K(G/P))\cong e\End_{KG}(K(G/B))e$. Note that for an arbitrary symmetric algebra $A$ and an idempotent $f$ in $A$, $fAf$ is also a symmetric algebra. Thus it suffices to show that $\mathcal{H}_K(G,B)\cong  \End_{KG}(K(G/B))$ is a symmetric algebra. As $p\neq q$ by assumption, Lemma \ref{C-H} combined with the fact that group algebras over arbitrary fields are symmetric algebras yields the assertion. \qed

\subsection{Criterion for $S$-permutation groups}

In this subsection, we mainly prove the following result.

\begin{Prop}\label{S-per}
Let $G$ be a transitive permutation group on a set $\Omega$ of cardinality $n$. Then $G$ is an $S$-permutation group if one of the following is satisfied

$(1)$ $G$ is a $\frac{3}{2}$-transitive permutation group;

$(2)$ $G$ has a cyclic regular subgroup of order a product of two distinct primes;

$(3)$ $G$ has an abelian regular subgroup for which the order is coprime to all subdegrees of $G$;

$(4)$ $G$ is a permutation group of subdegrees $1\leq a<b$ and $\gcd(a+b+1, ab, \lambda(G), b-\frac{\lambda(G)b}{a})=1$;

$(5)$ $G$ is a dihedral group of order not divisible by $8$;

$(6)$ $\Omega=G(q)/B$ and $G$ is the permutation group induced by the right multiplication of $G(q)$ on $\Omega$, where $G(q)$ is finite group with a split $BN$-pair at characteristic $q$ satisfying $\mathcal{W}_G\ncong S_3$ is rank 2.
\end{Prop}
Before we are able to prove Proposition \ref{S-per}, we prove several useful results.

\begin{Prop}\label{3/2-tran}
Any $\frac{3}{2}$-transitive permutation group is an $S$-permutation group.
\end{Prop}
{\bf Proof} Let $K$ be an arbitrary field of characteristic $p\geq 0$ and $G$ a $\frac{3}{2}$-transitive group on a set $\Omega$. By the main result in \cite{LM90}, $G$ is one of the following groups:

$(1)$ $2$-transitive groups;

$(2)$ Frobenius permutation groups;

$(3)$ affine primitive permutation groups with point stabilizers half-transitive linear groups;

$(4)$ almost simple groups: either

\quad$(a)$ $G=A_7$ or $S_7$ acts on the set of pairs in $\{1, \cdots, 7\}$, or;

\quad$(b)$ $\PSL(2,2^m)\leq G\leq {\rm P\Gamma L}(2,2^m)$  acts on $2^{m-1}(2^m-1)$ points and has rank $2^{m-1}$, where $m\geq 3$ is a prime.

Applying Lemma \ref{2-tr} and Lemma \ref{Frobenius} to case (1) and case (2), respectively, we deduce that $G$ is always an $S$-permutation group in these two cases.

Assume case (3), that is, $G=T(V)H \leq {\rm AGL}(V)$ for a vector space $V$, where $V$ is a vector space of dimension $d$ over a finite field $\mathbb{F}_p$, $T(V)$ represents the group of translations on $V$, $H$ is a subgroup of $\GL(V)$ and is a $\frac{1}{2}$-transitive irreducible linear group. Clearly, all non-trivial subdegrees are divisors of $|T(V)|-1$, and therefore are relatively prime to $|T(V)|$. Thus the condition of Proposition \ref{cop} is satisfied. This implies that  $G$ is an $S$-permutation group in this case.

Assume case (4). Then $G$ is one of the following groups:

\quad$(a)$ $G=A_{7}$ or $S_{7}$ acts on the set of pairs in $\{1, \cdots, 7\}$, or;

\quad$(b)$ $\PSL(2,2^m)\leq G\leq {\rm P\Gamma L}(2,2^m)$ acts on $2^{m-1}(2^m-1)$ points and has rank $2^{m-1}$, where $m\geq 3$ is a prime.

For case (a), it is well-known that the ranks of both groups equal $3$. According to Proposition \ref{rank}, we only need to consider the situation $G=S_7$ acting on set of pairs in $\{1, \cdots, 7\}$. In this case, we have $K\Omega = K_G\oplus M$, where $M$ is a $KG$-module. Moreover, $\End_{KG}(K\Omega)\simeq K\times \End_{KG}(M)$ is $3$-dimensional. This yields that $\End_{KG}(M)$ is two-dimensional. Thus $\End_{KG}(M)$ is a symmetric algebra. Hence $\End_{KG}(K\Omega)\simeq K\times \End_{KG}(M)$ is also a symmetric algebra. Since $K$ is arbitrary, it follows that $G$ is an $S$-permutation group in case (a).

For case (b), by Lemma \ref{rank}, it is sufficient to consider the case: $G=\PSL(2,2^m)$ acts on the cosets of a maximal subgroup, say $H$, of index $2^{m-1}(2^m-1)$. Then $H\cong \mathbb{D}_{2^m + 1}$, and all non-trivial subdegrees of $G$ are $2^m + 1$, that is, for any $x\in G - H$, we have $|H\cap H^x| = 2$. If $p$ does not divide $2^m + 1$, then it follows from Proposition \ref{Sym} that $\End_{KG}(K\Omega)$ is a symmetric algebra. Therefore, we may assume that $p\mid 2^m + 1$. In this case, $p\neq 0$ and is odd. Since $|H\cap H^x| = 2$ for any $x\in G - H$, it follows that $P\cap P^x = 1$ for all $x\in G - H$. Now, using a similar argument as in the proof of Proposition \ref{Frobenius}, we deduce that $\End_{KG}(K\Omega)$ is a symmetric algebra. \qed

\medskip
The following lemma shows that any $K$-Schur ring over a cyclic group of order $pq$ is a symmetric algebra.
\begin{Prop}\label{Sch}
Let $G$ be a cyclic group of order $pq$ for distinct primes $p$ and $q$. Then, for any field $K$, any $K$-Schur ring $\mathcal{A}$ over $G$ is always a symmetric algebra.
\end{Prop}
Before we prove Proposition \ref{Sch}, we need several lemmas.
\begin{Lem}\label{wedge}
Let $K$ be a field of characteristic $p$ and $\mathcal{A}$ be a $K$-Schur ring over a group $G$. If $\mathcal{A}$ is a wedge product $\mathcal{A}_H\wr\mathcal{A}_{G/H}$ for an $\mathcal{A}$-subgroup $H$  of order divided by $p$, and the two $K$-Schur rings $\mathcal{A}_H$ and $\mathcal{A}_{G/H}$ are symmetric algebras, then $\mathcal{A}$ is also a symmetric algebra.
\end{Lem}
\demo By hypothesis, there exist non-degenerate associative symmetric bilinear forms $f: \mathcal{A}_H\times \mathcal{A}_H\ra K$ and $g: \mathcal{A}_{G/H}\times \mathcal{A}_{G/H}\ra K$. Note that by definition basic sets $X\subseteq G\setminus H$ are union of cosets of $H$. Thus $HX=X$ for any basic set $X\subseteq G\setminus H$ and therefore we may also view these basic sets $X\subseteq G\setminus H$ as basic sets in $\mathcal{A}_{G/H}$. Now, we define a map $t:\mathcal{S}(\mathcal{A})\times \mathcal{S}(\mathcal{A})\ra K$ by the following rule:
$$t(\underline{X},\underline{Y})=\begin{cases}f(\underline{X},\underline{Y}) & \mbox{ if } X,Y\subseteq H,\\ g(\underline{X},\underline{Y}) & \mbox{ if } X,Y\nsubseteq H\\ |X|g(\underline{H},\underline{Y}) & \mbox{if } X\subseteq H, Y\nsubseteq H\\|Y|g(\underline{X},\underline{H}) & \mbox{if } X\nsubseteq H, Y\subseteq H.\end{cases}$$and extend $t$ bilinearly to $\mathcal{A}\times \mathcal{A}$.

We show that $t$ is a non-degenerate associative symmetric bilinear form. Clearly, $t$ is a non-degenerate symmetric bilinear form. It suffices to show that $t$ is associative. This is sufficient to show that for any basic sets $X, Y$ and $Z$, we have $t(\underline{X}\cdot\underline{Z}, \underline{Y})=t(\underline{X},\underline{Z}\cdot\underline{Y})$.

Suppose $X, Y\subseteq H$. Then, for $Z\subseteq H$, $t(\underline{X}\cdot\underline{Z}, \underline{Y})=t(\underline{X},\underline{Z}\cdot\underline{Y})$ follows immediately from the associativity of $f$, and for $Z\nsubseteq H$, $t(\underline{X}\cdot\underline{Z}, \underline{Y})=|X||Y|g(\underline{H}, \underline{Z})=t(\underline{X},\underline{Z}\cdot\underline{Y})$.

Suppose either $X\subseteq H, Y\nsubseteq H$, or $X\nsubseteq H, Y\subseteq H$. We only consider the case $X\subseteq H$ and $Y\nsubseteq H$, as the other case is similar. If $Z\subseteq H$, then
$$t(\underline{X}\cdot\underline{Z}, \underline{Y})=|X||Z|g(\underline{H}, \underline{Y})=t(\underline{X},\underline{Z}\cdot\underline{Y}).$$
Thus we may assume $Z\nsubseteq H$. If $Z\neq Y^{-1}$, then
$$t(\underline{X}\cdot\underline{Z}, \underline{Y})=|X|g(\underline{Z}, \underline{Y})=|X|g(\underline{H}, \underline{Z}\cdot\underline{Y})=t(\underline{X},\underline{Z}\cdot\underline{Y}).$$
Hence we may further assume that $Z=Y^{-1}$. Therefore,
$$
\begin{aligned}
t(\underline{X},\underline{Z}\cdot\underline{Y}) &= t(\underline{X}, \underline{Z}\cdot\underline{Y}-|Z|\underline{H}) + t(\underline{X},|Z|\underline{H}) \\
&= |X|g(\underline{H}, \underline{Z}\cdot\underline{Y}-|Z|\underline{H}) + f(\underline{X},|Z|\underline{H}).
\end{aligned}
$$
Since $Z$ is by assumption a union of $H$-cosets, $|H|$ divides $|Z|$. The assumption $p\mid |H|$ then implies $|Z|\cdot 1_K=0$. Consequently,
$$
\begin{aligned}
t(\underline{X},\underline{Z}\cdot\underline{Y}) &= |X|g(\underline{H}, \underline{Z}\cdot\underline{Y}-|Z|\underline{H}) + f(\underline{X},|Z|\underline{H}) \\
&= |X|g(\underline{H}, \underline{Z}\cdot\underline{Y}) \\
&= t(\underline{X}\cdot\underline{Z}, \underline{Y}).
\end{aligned}
$$

Finally, suppose $X \nsubseteq H$ and $Y \nsubseteq H$. If $Z \notin \{X^{-1}, Y^{-1}\}$, we consider two subcases:

If $Z \subseteq H$, then $t(\underline{X} \cdot \underline{Z}, \underline{Y}) = |Z| \, g(\underline{X}, \underline{Y}) = t(\underline{X}, \underline{Z} \cdot \underline{Y})$.

 If $Z \nsubseteq H$, the equality $t(\underline{X} \cdot \underline{Z}, \underline{Y}) = t(\underline{X}, \underline{Z} \cdot \underline{Y})$ follows directly from the associativity of $g$.

Thus we may assume $Z \in \{X^{-1}, Y^{-1}\}$. Without loss of generality, we take $Z = X^{-1}$; the case $Z = Y^{-1}$ is handled similarly.
Using $|Z| \cdot 1_K = 0$, we obtain
\[
\begin{aligned}
t(\underline{X} \cdot \underline{Z}, \underline{Y})
&= t(|Z|\underline{H}, \underline{Y}) + g(\underline{X} \cdot \underline{Z} - |Z|\underline{H}, \underline{Y}) \\
&= |Z|\,|H|\, g(\underline{H}, \underline{Y}) + g(\underline{X} \cdot \underline{Z} - |Z|\underline{H}, \underline{Y}) \\
&= g(\underline{X} \cdot \underline{Z}, \underline{Y}).
\end{aligned}
\]
Similarly, one finds $t(\underline{X}, \underline{Z} \cdot \underline{Y}) = g(\underline{X}, \underline{Z} \cdot \underline{Y})$, regardless of whether $Z = Y^{-1}$.
By the associativity of $g$, we conclude that
\[
t(\underline{X} \cdot \underline{Z}, \underline{Y}) = t(\underline{X}, \underline{Z} \cdot \underline{Y}).
\] \qed
\medskip

The following result shows that any cyclotomic $K$-Schur ring over a cyclic group of square-free order is a symmetric algebra.
\begin{Prop}\label{Cyc}
Let $K$ be a field. Then any cyclotomic $K$-Schur ring over a cyclic group of square-free order is a symmetric algebra.
\end{Prop}
\demo For a group $G$ and an automorphism $\sigma\in \Aut(G)$, $\sigma$ induces a natural algebra automorphism on the group algebra $KG$, by sending $\sum_{g\in G}a_g g$ to $\sum_{g\in G}a_g g^\sigma$. In this sense, a cyclotomic $K$-Schur ring over a group $G$ is the subalgebra $(KG)^H$ of $KG$ for a group of automorphism $H\leq \Aut(G)$.

Let $K$ be a field of characteristic $p$, and let $G$ be a cyclic group of square-free order and $H\leq \Aut(G)$ a group of automorphism. Clearly, the statement is true for $p\nmid |G|$. Thus we may assume that $p>0$ and $p\mid |G|$. Note that $H$ is abelian. Let $P$ and $Q$ be the Sylow $p$-subgroup of $H$ and $G$, respectively. Then $H=P\times R$ is a direct product of $P$ and a $p'$-subgroup $R$ and $G=Q\times T$ is a direct product of $Q$ and a $p'$-subgroup $T$. Then $(KG)^H=((KG)^P)^R$, where $((KG)^P)^R$ is well defined as $P$ and $R$ commutes and therefore $(KG)^P$ admits an $R$-action. Note that $(KG)^H\cong (\mathbb{F}_p G)^H\otimes_{\mathbb{F}_p} K$. Thus if we could show that $B:=((\mathbb{F}_pG)^P)^R$ is a symmetric algebra, then it would follow that $(KG)^H$ is also a symmetric algebra. Let $K_0/\mathbb{F}_p$ be a extension of $\mathbb{F}_p$ such that $\End_{K_0G}(K_0\Omega)\cong \End_{\mathbb{F}_pG}(\mathbb{F}_p\Omega)\otimes_{\mathbb{F}_p} K_0$ and $K_0T$ are split algebras (a split algebra is an algebra for which the endomorphism algebras of all simple modules are isomorphic to the ground field).

We first show that $A:=(K_0G)^P$ is a symmetric algebra. Note that $Q$ is a characteristic subgroup of $G$ of order $p$, and therefore has no automorphisms of order $p$. Thus the restriction of $P$ to $Q$ is trivial. Hence $A^P\cong K_0Q\otimes_{K_0} (K_0T)^P$. Since $T$ is an abelian $p'$-subgroup, it follows from our assumption on $K_0$ that $K_0T$ is isomorphic to a direct product of $m$ copies of $K_0$. This implies that $(K_0T)^P$ is isomorphic to a direct product of $k$ copies of $K_0$ for some $k\leq m$. Thus $A^P\cong K_0Q\otimes_{K_0} (K_0T)^P$ is isomorphic to a direct product of $k$ copies of $(K_0[x]/(x^p))$ and therefore is a symmetric Nakayama algebra.

Next, we show $C:=(K_0G)^H=A^R$ is a symmetric algebra. We consider the skew group algebra, denoted $R\ltimes A$, which is a free $A$-module consisting elements of the form $ra$ for $ r\in R, a\in A$, and the multiplication in $R\ltimes A$ is given by the rule $(r_1 a_1)\cdot (r_2 a_2):=r_1r_2a^{r_2}_1 a_2$. Now, we endow a $R\ltimes A$-module structure on $A$ by defining $a\cdot (rb):=a^rb$. Then one can get an isomorphism $\End_{R\ltimes A}(A)\cong A^R, f\mapsto 1^f$. Since $R$ is a $p'$-group, it follows from \cite[Corollary 4.3(a), p. 87 and Theorem 2.14, p. 118]{ARS} that $A$ is a projective $R\ltimes A$-module and $R\ltimes A$ is a Nakayama algebra. Note that $C$ is commutative. Then the tensor functor $-\otimes_C{} _CA_{R\ltimes A}:C\modcat\ra (R\ltimes A)\modcat$  is a fully faithful embedding of categories. Hence $C=A^R$ is a commutative split algebra of finite representation type. This implies that there exist non-negative integers $l_1,l_2, \cdots, l_s$ such that $A^R\cong \prod^{s}_{j=1}K_0[x]/(x^{l_j})$. Thus $A^R$ is a symmetric algebra.

For any $B$-module $M$, we have $$\Ext^1_B(M,B)\otimes_{\mathbb{F}_p} K_0\cong \Ext^1_C(M\otimes_{\mathbb{F}_p} K_0, C)=0.$$ This implies that $B$ is a self-injective algebra. Since $B$ is commutative and any commutative self-injective algebra are symmetric algebra, the statement follows. \qed

\medskip
Now, we are in a position to prove Proposition \ref{Sch}.

\medskip
\noindent {\bf Proof of Proposition \ref{Sch}.} Let $K$ be a field and let $\mathcal{A}_G$ be a $K$-Schur ring over a cyclic group $G$ of order $pq$ for distinct primes $p$ and $q$. Clearly, if the characteristic of $K$ does not divide $pq$ then $\mathcal{A}_G$ is semisimple and therefore symmetric. Without loss of generality, we may assume that the characteristic of $K$ is $p$. By Proposition \ref{C-S} we have one of the following:

$(i)$ $\mathcal{A}$ is trivial;

$(ii)$ $\mathcal{A}$ is a cyclotomic $K$-Schur ring;

$(iii)$ there exist $\mathcal{A}$-subgroups $H$ and $U$ such that $G=H\times U$ and $\mathcal{A}=\mathcal{A}_H\times\mathcal{A}_U$ is a tensor product of two $K$-Schur rings $\mathcal{A}_H$ and $\mathcal{A}_U$;

$(iv)$ there exist proper $\mathcal{A}$-subgroups $U\leq H\leq G$ such that $\mathcal{A}=\mathcal{A}_H\wr_{H/U}\mathcal{A}_{G/U}$ is a wedge product of two $K$-Schur rings $\mathcal{A}_H$ and $\mathcal{A}_{G/U}$.

In particular, $K$-Schur ring over a cyclic group of prime order are necessarily cyclotomic. For case (i), $\mathcal{A}_G$ is either the group algebra $KG$ or a 2-dimensional algebra, and therefore is always a symmetric algebra in any cases. For case (ii), it follows from Proposition \ref{Cyc} that $\mathcal{A}_G$ is a group algebra. Assume case (iii). Then $\mathcal{A}_G=\mathcal{A}_H\times\mathcal{A}_U$ is a tensor product of two $K$-Schur rings  $\mathcal{A}_H$ and $\mathcal{A}_U$, where $H$ and $U$ cyclic groups of order $p$ and $q$, respectively. By Proposition \ref{Cyc} again, we see that $\mathcal{A}_H$ and $\mathcal{A}_U$ are symmetric algebras and so are their tensor product. This implies that $\mathcal{A}_G$ is a symmetric algebra.

Now, assume case (iv). Then it follows from $|G|=pq$ that $\mathcal{A}=\mathcal{A}_H\wr\mathcal{A}_{G/H}$ is a wedge product of two $K$-Schur rings $\mathcal{A}_H$ and $\mathcal{A}_{G/H}$. Assume $|H|=p$. Since $\mathcal{A}_H$ and $\mathcal{A}_{G/H}$ are symmetric algebras, it follows from Lemma \ref{wedge} that $\mathcal{A}_G$ is a symmetric algebra. Thus we may assume that $|H|=q$. Now, we define a map $f:\mathcal{A}_G \times \mathcal{A}_G\ra K$ by letting $f(x,y)$ be the coefficient of $1$ in the product of $xy$. It is straightforward to check that $f$ is indeed a non-degenerate associative symmetric bilinear form. Hence $\mathcal{A}_G$ is a symmetric algebra. This completes the proof.\qed
\vskip 2mm
The following lemma gives a criterion for whether a rank $3$ permutation group is an $S$-permutation group or not.
\begin{Lem}\label{rank3}
Let $G$ be a rank $3$ permutation group of subdegrees $1\leq a<b$. Then $G$ is an $S$-permutation group if and only if $\gcd(a+b+1, ab, \lambda(G), b-\frac{\lambda(G)b}{a})=1$.
\end{Lem}
\demo Let $H$ be a point stabilizer in $G$. For convenience, we denote $\lambda(G)$ by $\lambda$.

Suppose $\gcd(a+b+1, ab, \lambda, \frac{\lambda b}{a})=1$. We need to show that $\mathcal{H}_K(G,H)\cong\End_{KG}(K(G/H))$ is a symmetric algebra for an arbitrary field $K$ of characteristic $p\geq 0$. By Proposition \ref{p-S}, we may assume that $p\mid ab$ and $p\mid (1+a+b)$. In particular, $p\nmid \gcd(a,b)$. Let $H, D_1=HgH$ and $D_2=HtH$ be the three distinct $(H,H)$-double cosets such that $\frac{|D_1|}{|H|}=a$ and $\frac{|D_2|}{|H|}=b$.

To prove that $\mathcal{H}_K(G,H)$ is a symmetric algebra, it suffices to show that $\mathcal{H}_K(G,H)\ncong K[x,y]/(x^2,xy,y^2)$. Suppose for the contrary that $\pi: \mathcal{H}_K(G,H)\ra K[x,y]/(x^2,xy,y^2)$ is an isomorphism. Then $\pi(H)=1+(x^2,xy,y^2)$, $\pi(D_1)=l_1+k_1x+s_1y+(x^2,xy,y^2)$ and $\pi(D_2)=l_2+k_2x+s_2y+(x^2,xy,y^2)$ for $l_i, k_i, s_i\in K$, and these three elements form a basis of $K[x,y]/(x^2,xy,y^2)$. In particular, $(k_1,s_1)\neq (0,0)\neq (k_2,s_2)$. Thus
$$\begin{array}{lll}& \pi(D^2_1) = l^2_1+2l_1k_1 x+2l_1 s_1 y+(x^2,xy,y^2),\\
& \pi(D^2_2) = l^2_2+2l_2k_2 x+2l_2 s_2 y+(x^2,xy,y^2),\\
& \pi(D_1\cdot D_2)=l_1l_2+(l_1k_2+l_2k_1) x+(l_1 s_2+l_2s_1) y+(x^2,xy,y^2).
\end{array}$$On the other hand, the multiplicative rule for $D_1$ and $D_2$ in Remark \ref{rank-3} implies that

$$\begin{array}{lll}& \pi(D^2_1) = a\cdot 1_K+(a-1-\frac{\lambda b}{a})\pi(D_1)+\lambda \pi(D_2),\\
& \pi(D^2_2) = b\cdot 1_K+(b-\frac{\lambda b}{a})\pi(D_1)+(b-a-1+\lambda)\pi(D_2),\\
& \pi(D_1\cdot D_2)= \pi(D_2\cdot D_1)=\frac{\lambda b}{a}\cdot \pi(D_1)+(a-\lambda)\pi(D_2)
\end{array}$$Comparing these equalities with the equalities above, we find that $ \{l_1,l_2\}=\{-1_K,0\}$ and $\lambda\cdot 1_K= (b-\frac{\lambda b}{a})\cdot 1_K=0$. This together with $p\mid (a+b+1)$ and $p\mid ab$ implies that $p\mid \gcd(a+b+1, ab, \lambda, b-\frac{\lambda b}{a})=1$, a contradiction.

Conversely, suppose $\gcd(a+b+1, ab, \lambda, b-\frac{\lambda b}{a})>1$. Let $p$ be a divisor of $\gcd(a+b+1, ab, \lambda, b-\frac{\lambda b}{a})>1$ and let $K$ be a field of characteristic $p$. Then by the argument above we  derive that $\mathcal{H}_K(G,H)\cong K[x,y]/(x^2,xy,y^2)$. Therefore, $\mathcal{H}_K(G,H)$ is not a symmetric algebra in this case. This completes the proof. \qed
\begin{rem}\label{rank-3}
Let $G$ be a rank $3$ transitive permutation group and let $1\leq a<b$ be the subdegrees of $G$. For a point stabilizer $H$ in $G$, let $H, D_1=HgH$ and $D_2=HtH$ be the three distinct $(H,H)$-double cosets such that $\frac{|D_1|}{|H|}=a$ and $\frac{|D_2|}{|H|}=b$. Then by \cite[Proposition 3.3, p.42]{AK}, there exists $\lambda(G)=\lambda\in \mathbb{N}$ such that in $\mathcal{H}_{\mathbb{Z}}(G,H)$ the following holds:
$$\begin{array}{lll}& D^2_1 = aH+(a-1-\frac{\lambda b}{a})D_1+\lambda D_2,\\
& D^2_2 = bH+(b-\frac{\lambda b}{a})D_1+(b-a-1+\lambda)D_2,\\
& D_1\cdot D_2= D_2\cdot D_1=\frac{\lambda b}{a}\cdot D_1+(a-\lambda)D_2.
\end{array}$$where the coefficients appearing in these equalities are non-negative integers.
\end{rem}
Now, we are able to prove Proposition \ref{S-per}.

\medskip
\noindent {\bf Proof of Proposition \ref{S-per}.} Let $K$ be a field of characteristic $p$.

For case (1), we deduce from Proposition \ref{3/2-tran} that $G$ is an $S$-permutation group. For case (2), Lemma \ref{regul} and Proposition \ref{Cyc} imply that $G$ is an $S$-permutation group. For case (3), it follows from Proposition \ref{p-S}(i)(ii) that $G$ is an $S$-permutation group.

For case (4), by Lemma \ref{rank3} we see that $G$ is an $S$-permutation group. Assume case (5). Then $G$ is either a regular permutation group or a permutation group with a point stabilizer of order two. Clear, a regular permutation group is an $S$-permutation group. Then we may assume that $G$ is a permutation group with a point stabilizer of order two. According to (1), we only need to show that $\End_{KG}(K\Omega)$ is a symmetric algebra when $p=2$. Note that $G$ has a normal cyclic regular subgroup, say $H$, of order not divisible by $4$. Let $U\leq \Aut(H)$ be the subgroup generated by the automorphism sending all $h\in H$ to $h^{-1}$. According to Lemma \ref{regul}, it follows that $\End_{KG}(K\Omega)\cong (KH)^U$. Let $H=T\times L$ be a direct product of a Sylow $2$-subgroup $T$ and a $2'$-subgroup $L$. Then $|T|\leq 2$. With a similar argument as in the proof of Proposition \ref{Cyc}, we see that $(KH)^U\cong KT\otimes_K (KL)^U$. Clearly, both of $KT$ and $(KL)^U$ are symmetric algebras. This yields that $\End_{KG}(K\Omega)\cong(KH)^U$ is a symmetric algebra.

Assume (v). Set $$r_1(x,y)=xy-yx, r_2(x,y)=(xy)^2-(yx)^2 ~\mbox{and}~ r_3(x,y)=(xy)^3-(yx)^3.$$ By Proposition \ref{p-S}(v), we only need to show that $\mathcal{H}_K(G,B)$ is a symmetric algebra when $p=q$. Let $\mathcal{G}$ be the Coxeter graph associated with the $\mathcal{W}_G$. Then by the main theorem in \cite{Iwa}, we obtain $$\mathcal{H}_K(G,B)\cong K\lg x,y\rg/(r_i(x,y),x^2+x,y^2+y),$$ where $i=1, i=2$ and $i=3$ according to $\mathcal{G}=D_2, B_2(\mbox{or}~C_2)$ and $G_2$ (here we use the standard notion in Coxeter graphs), respectively. For $i\in \{1,2,3\}$, set $$\small{A_i:=K\lg x,y\rg/(r_i(x,y),x^2+x,y^2+y), x_i:=x+(r_i(x,y),x^2+x,y^2+y)~\mbox{and}~ y_i:=y+(r_i(x,y),x^2+x,y^2+y).}$$ By a word of $x_i$ and $y_i$ we mean an element of form $x_iy_ix_i\cdots$ or $y_ix_iy_i\cdots$. Then words in $x_i$ and $y_i$ of length at most $2i$ form a basis, say $B_i$, of $A_i$.

For $1\leq i\leq 3$ and any two words $M_1$ and $M_2$ of $x_i$ and $y_i$, we define $t(M_1, M_2)$ to be the coefficient of $(x_iy_i)^i$ in the expression of $M_1M_2$ as a linear combination of the elements in $B_i$ and extend this bilinearly to $A_i\times A_i$. Then it is directly to check that $t$ is indeed a non-degenerate associative symmetric bilinear form. Thus $K\lg x,y\rg/(r_i(x,y),x^2+x,y^2+y)$ is a symmetric algebra for $1\leq i\leq 3$. \qed
\section{Examples}\label{exam}
In this section, we present two examples to illustrate our main results.

\medskip
Although the Hecke algebras associated with rank $3$ permutation groups are $3$-dimensional, but there exist rank $3$ permutation groups that are not $S$-permutation groups.

\begin{exam}
{\rm Let $G\leq S_4$ be the subgroup generated by the permutations $(1234)$ and $(24)$. Then $G$ is dihedral subgroup of order $8$, which is transitive on $\Omega=\{1,2,3,4\}$ and of rank $3$ with subdegrees $1,1$ and $2$. Moreover, we have $\lambda(G)=0$ and therefore $\gcd(4, 2, 0, 2)=2>1$. Let $K$ be a field of characteristic $2$. By computation, we see that $\End_{KG}(K\Omega)\cong K[x,y]/(x^2,xy,y^2)$ is not a symmetric algebra. This explains the conditions in Proposition \ref{p-S}(iii) and Proposition \ref{S-per}(4)-(5), and this also shows that a permutation group with a cyclic regular subgroup of order a prime square is not necessarily an $S$-permutation group. }
\end{exam}
The following example shows that $\mathcal{H}_K(G,B)$ is not a symmetric algebra when $G$ is finite group with a split $BN$-pair at characteristic $q$ satisfying $\mathcal{W}_G\cong S_3$ and the characteristic of $K$ equals $q$.
\begin{exam}{\rm
Let $G=\GL(3,q)$ and $B$ a Borel subgroup, and let $K$ be a field of characteristic of $q$. Then by generators and relations described in \cite{Iwa} we have $$\mathcal{H}_K(G,B)\cong K\lg x,y\rg/(xyx-yxy,x^2+x,y^2+y).$$ Set $A:=K\lg x,y\rg/(xyx-yxy,x^2+x,y^2+y)$, and let $u, v\in A$ be the cosets of $x$ and $y$, respectively. Then $-u, -v$ and $uvu$ are idempotents in $A$ and $uvu$ is central. Thus $P_1:=-uA=uA, P_2:=-vA=vA$ and $P_3:=uvuA$ are projective $A$-modules. Since $P_3$ is one dimensional, it is simple. This together with $\Hom_A(P_1,P_3)\cong uvuAu\neq 0$ and $\Hom_A(P_2,P_3)\cong uvuAv\neq 0$ implies that $P_3\mid P_k$ for $k=1,2$. Now, we write $P_k=M_k\oplus P_3$ as direct sum of a submodule $M_k$ and $P_1$ for $k=1,2$. Comparing the dimensions of $uA+vA$ and $M_1+M_2+P_3$, we find that $uA+vA=M_1\oplus M_2\oplus P_3$. Thus there exists a simple projective $A$-module $Q$ such that $A=M_1\oplus M_2\oplus P_3\oplus Q$. Note that $$\Hom_{A}(P_1,P_2)\cong vAu, \quad \Hom_{A}(P_2,P_1)\cong uAv, \quad \End_{A}(P_1)\cong uAu, ~\mbox{and}~  \End_{A}(P_2)\cong vAv.$$By these isomorphisms, we find that $\End_A(M_k) (k=1,2), \Hom_A(M_1,M_2)$ and $\Hom_A(M_2,M_1)$ are $1$-dimensional and $\End_A(uA+vA)$ is $5$-dimensional. Since $A\opp\cong \End_A(A)$ is $6$-dimensional, it follows that $\Hom_A(Q, uA+vA)=0=\Hom_A(uA+vA,Q)$. Then $A\opp\cong \End_A(A)$ is isomorphic the algebra $K\times K\times \Lambda$, where $\Lambda=\End_A(M_1\oplus M_2)$ is given by the following quiver with relation $$\xymatrix{
	\bullet  \ar@<-0.4ex>[r]_{\alpha}_(0){1}_(1){2}
	& \bullet \ar@<-0.4ex>[l]_\beta}, \qquad \alpha\beta=\beta\alpha=0.
$$Thus $A\cong A\opp$ as algebras. Note that the head and socle of any indecomposable projective module over a symmetric algebra are isomorphic simple modules. From this fact, we see that $A$ is not a symmetric algebra.}
\end{exam}
\section{Further questions}\label{ques}
In this section, we propose some questions for further research.

\medskip
By Proposition \ref{rank}, a subgroup of an $S$-permutation group with the same rank is also an $S$-permutation. We say an $S$-permutation group is {\it minimal} if it does not have any subgroup of the same rank.

\begin{ques}
Can we classify all minimal $S$-permutation groups?
\end{ques}
Primitive permutation groups paly an important role in permutation group theory. To be concerned with this class of permutation groups, the following question arises
\begin{ques}
Can we classify all primitive $S$-permutation groups?
\end{ques}
Motivated by Lemma \ref{regul}, if $H$ is a finite group for which all $K$-Schur rings over $H$ are symmetric algebras, then any permutation containing $H$ as a regular subgroup is an $S$-permutation group. Thus the following question arises
\begin{ques}
Can we determine all finite group $H$ for which all $K$-Schur rings over $H$ are symmetric algebras?
\end{ques}
\noindent{\bf Data Availability Statement.}
Data sharing is not applicable to this article as no new data were created or analyzed in this study.

\medskip

\noindent{\bf Acknowledgement.}

The first author was partially supported by the Natural Science Foundation of Jiangxi Province, China (Grant No. 20252BAC200143) and the National Natural Science Foundation of China (Grant No. 12501025) . The second author was partially supported by the National Natural Science Foundation of China (Grant Nos. 12350710787, 12571016).

\end{document}